\documentclass[dvipdfmx,12pt,a4paper,reqno]{amsart}
\usepackage{preamble}

\begin{document}

\title[Reaction--diffusion model]{Exponentially slow mixing and hitting times of rare events for a reaction--diffusion model}
\author[K. Tsunoda]{Kenkichi Tsunoda}
\address{Department of Mathematics, Osaka university, Toyonaka, Osaka 560-0043, JAPAN and RIKEN Center for Advanced Intelligence Project (AIP), Tokyo, JAPAN}
\email{k-tsunoda@math.sci.osaka-u.ac.jp}
\date{\today}

\maketitle

\begin{abstract}
We consider the superposition of symmetric simple exclusion dynamics speeded-up in time,
with spin-flip dynamics in a one-dimensional interval with periodic boundary conditions.
We show that the mixing time has an exponential lower bound in the system size
if the potential of the hydrodynamic equation has more than two local minima.
We also apply our estimates to show that the normalized hitting times of rare events
converge to a mean one exponential random  variable if the potential has a unique minimum.
\end{abstract}

\section{Introduction}\label{sec1}
In this paper, we study the superposition of symmetric simple exclusion dynamics speeded-up in time,
with spin-flip dynamics in a one-dimensional interval with periodic boundary conditions.
We call this model the {\it reaction--diffusion model}\footnote{This model is also known as {\it Glauber+Kawasaki model}.}.
De Masi, Ferrari, and Lebowitz in \cite{MR857069} have introduced this model
to study a reaction--diffusion equation of the form
\begin{align}\label{sec1hdleq}
\partial_t\rho=(1/2)\Delta\rho - V'(\rho),
\end{align}
where $V$ is a potential from a stochastic microscopic systems viewpoint.
They showed the {\it hydrodynamic limit}, that is, the macroscopic density of the reaction--diffusion model
evolves according to the reaction--diffusion equation \eqref{sec1hdleq}, under diffusive scaling. 
We refer  to \cite{MR1175626} and \cite[Subsection 3.1]{MR3956160} and the references therein
for the recent development of the reaction--diffusion model.

This paper is a continuation of our studies \cite{MR3765880,MR3947320,tanaka2020glauberexclusion}
and  we use several results established in these papers.
We have studied the hydrostatic limit and the dynamical large deviation principle in \cite{MR3765880},
the static large deviation principle in \cite{MR3947320},
and rapid mixing in \cite{tanaka2020glauberexclusion}.
More precisely, in \cite{tanaka2020glauberexclusion}, we have shown that
the total variation mixing time is of the order $\log N$ ($N$ is the system size)
if the reaction--diffusion model is {\it attractive} and the hydrodynamic equation \eqref{sec1hdleq}
has a strictly convex potential.
Therefore, it is natural to ask what happens when the potential $V$ has more than two local minima.

We first consider the case where the potential $V$ has more than two local minima
and show that the total variation mixing time is bounded below by $e^{cN}$ for some constant $c$
for any $N$ sufficiently large. In particular, for the case of the original model introduced in \cite{MR857069},
our result and rapid mixing established in \cite{tanaka2020glauberexclusion}
imply a phase transition for the mixing time. Namely, if an inverse temperature of the system
is larger than some critical temperature,
the reaction--diffusion model exhibits exponentially slow mixing, otherwise rapid mixing.
Note that this type of phase transition cannot be observed
for the Glauber dynamics related to the Ising model
on the one-dimensional periodic domain \cite[Theorem 15.5]{MR3726904}.
This is surprising because our particle system interacts only through nearest neighbors,
and the dimension of the underlying system is one.

Using hitting time estimates, which will be established in this paper,
we study the hitting times of rare events for the reaction--diffusion model when the potential has a unique minimum.
As the second main result, we show that the hitting time of an open set,
which does not contain a unique minimum of the potential,
converges to a mean one exponential random variable.
We note that the techniques developed in this paper are robust enough to apply to other models,
including boundary--driven exclusion processes \cite{MR1997915,MR2052137,farfan2009static,MR2770905}.

We  mention several papers related to this work.
Our motivation to study this problem originates from two recently developed theories.
One is the {\it macroscopic fluctuation theory} and the other is the {\it martingale approach to metastability}.
For details of each theory, see survey papers \cite{MR3403268} and \cite{MR3960293}, respectively.
This paper combines these two theories following the Freidlin--Wentzell theory \cite{MR1652127}.
Apparently, the mixing time for the exclusion process or Glauber dynamics is related to the problem we consider.
Lacoin et al. have extensively studied the mixing time for the exclusion process
\cite{MR2869447,MR3474475,MR3551201,MR3689972,MR3945753,MR4132639}.
The mixing time for the Glauber dynamics has been classically studied.
We only refer to \cite{MR3726904}  and the sophisticated work by Lubetzky and Sly \cite{MR3020173}.
The convergence to a mean one exponential random variable also has a long history in probability theory \cite{MR528293}.
As clarified later, we use a general criterion established in \cite{MR3131505}.
Therefore, we also refer to the references in this paper.
We finally mention Hinojosa's study. He has studied the convergence to a mean one exponential random variable
of an exit time for the reaction--diffusion model on the entire domain
in a double-well case \cite{MR2127751} and a one-well case \cite{MR3845029}.

This paper is organized as follows. We introduce our model and results in Section \ref{sec2}.
We prove one of our main results (Theorem \ref{thm2-1}) at the end of this section.
Section \ref{sec3} is devoted to proving Lemma \ref{lem2-3}, which is critical in proving Theorem \ref{thm2-1}.
In Section \ref{sec4}, we prove our second main result (Theorem \ref{thm4-1}).
Since our argument strongly relies on the results in \cite{MR3765880,MR3947320},
we summarize several results for reader's convenience in the appendix.
Appendix \ref{seca} collects miscellaneous properties about the reaction--diffusion equation
and Appendix \ref{secb} discusses the rate function of the dynamical large deviation principle.

\section{Notation and Results}\label{sec2}

Let $\T_N=\Z/N\Z$, $N\ge 1$, be a one-dimensional discrete torus
with $N$ points.  Denote the set $\{0,1\}^{\T_N}$ by $X_N$ and 
the elements of $X_N$ by $\eta$, called {\it configurations}.  For each
$x\in\T_N$ and $\eta\in X_N$,
$\eta(x)$ represents the occupation variable at site $x$
so that $\eta(x)=1$ if site $x$ is occupied,
and $\eta(x)=0$ if site $x$ is vacant.  For each $x \not
=y\in\T_N$, denote by $\eta^{x,y}, \eta^{x}$, the
configuration obtained from $\eta$ by exchanging the occupation
variables $\eta(x)$ and $\eta(y)$, by flipping the occupation
variable $\eta(x)$, respectively.
\begin{align*}
\eta^{x,y}(z)  =  \begin{cases}
     \eta(y)  & \text{if $z=x$} , \\
     \eta(x)  & \text{if $z=y$} , \\
     \eta(z)    & \text{otherwise} ,
\end{cases}
\quad
\eta^x(z)  =  \begin{cases}
     1-\eta(x)  & \text{if $z=x$} ,  \\
     \eta(z)  & \text{if $z\neq x$} .
\end{cases}
\end{align*}

Consider a superposition of the speeded-up symmetric simple
exclusion process with spin-flip dynamics. The generator of this
$X_N$-valued, continuous-time Markov process acts on functions
$f:X_N\to\R$ as
\begin{align*}
L_Nf  = L_Gf + N^2 L_Kf ,
\end{align*}
where $L_G$ is the generator of spin-flip dynamics (Glauber dynamics)
\begin{align*}
L_Gf(\eta)  =  \sum_{x\in\T_N} c(x,\eta)[f(\eta^x) - f(\eta)] ,
\end{align*}
and $L_K$ is the generator of a symmetric simple exclusion
process (Kawasaki dynamics)
\begin{align*}
L_Kf(\eta)  =  (1/2)\sum_{x\in\T_N} [f(\eta^{x,x+1}) - f(\eta)] .
\end{align*}
In defining $L_G$, the jump rate $\{ c(x,\eta): x\in\T_N,\eta\in X_N\}$ is chosen as
$c(x,\eta)=c(\eta(\cdot+x))$ for a given function $c:\{0,1\}^{\Z}\to[0,\infty)$, where modulo $N$ carries the sum.
We also assume that $c$ is {\it local} in the sense that $c$ depends only on finitely many occupation variables $\eta(x)$.
Then, $c$ is identified with a function on $X_N$ for $N$ sufficiently large.

In this paper, we always assume that $c$ is strictly positive,
assuring that the Markov process generated by $L_N$ is irreducible.
Therefore, the process admits a unique probability distribution, which is invariant under the dynamics.
We denote by $\mu_{N}$ its unique stationary probability measure. 

Fix a topological space $X$. For $I=[0,T]$, $T>0$, or $I = \R_+=[0,\infty)$, let $C(I,X)$ be the space of
continuous trajectories from $I$ to $X$, endowed with the uniform topology.
Similarly, let $D(I,X)$ be the space of right continuous trajectories from $I$
to $X$ with left limits, endowed with the Skorokhod topology.  For each $N$, let
$\{\eta_t^N:t\ge0\}$ be the continuous-time Markov process on $X_N$
whose generator is given by $L_N$. For a probability
measure $\nu$ on $X_N$, denote by ${{\bb P}}_{\nu}$ the probability measure
on $D(\R_{+},X_N)$ induced by the process $\eta^{N}_{t}$ starting
from $\nu$.  Denote the measure ${{\bb P}}_{\nu}$ by ${{\bb P}}_{\eta}$ when
the probability measure $\nu$ is the Dirac measure concentrated on the
configuration $\eta$. The expectation with respect to ${{\bb P}}_{\eta}$ is represented
by $\bb E_{\eta}$.

Let $\nu_\rho=\nu^N_{\rho}$, $0\le\rho\le 1$, be the Bernoulli product
measure on $X_N$ with a density $\rho$.  Define the polynomial functions $B,D:[0,1]\to\R$ by
\begin{align*}
B(\rho) = \int [1-\eta(0)] c(0,\eta) d\nu_\rho , \quad
D(\rho) = \int \eta(0) c(0,\eta) d\nu_\rho  .
\end{align*}
We also set $F(\rho)=B(\rho)-D(\rho)$ and denote
a primitive function of $-F$ by $V$.
We call $V$ a {\it potential}.
Note that $V$ has at least one local minimum on $(0,1)$ since 
$F(0)>0, F(1)<0$, and $V(\rho)$ is a polynomial in $\rho$.

As examined in the introduction, De Masi, Ferrari, and Lebowitz in \cite{MR857069} has shown that
under an appropriate convergence of the initial distribution, the macroscopic density 
\begin{align*}
\pi_t^N=\frac{1}{N}\sum_{x\in \T_N}\eta_t^N(x)\delta_{x/N} ,
\end{align*}
converges in probability to a unique weak solution to  the reaction--diffusion equation
\begin{align*}
\partial_t\rho=(1/2)\Delta\rho + F(\rho).
\end{align*}
As clarified later, our proof strongly relies on the corresponding large deviation principle (Theorem \ref{thm2-2}).

We here give an example of the jump rate $c$.
The following example has been given in \cite{MR857069}.

\begin{example}\label{exa2-1}
For $0\le\gamma<1$, define
\begin{align*}
c(\eta)=1+\gamma(1-2\eta(0))(\eta(1)+\eta(-1)-1)+\gamma^2(2\eta(-1)-1)(2\eta(1)-1).
\end{align*}
Letting $\gamma=\tanh \beta, \beta\ge0$, the Glauber dynamics generated by $L_G$ is
reversible with respect to a Gibbs measure of the one-dimensional nearest neighbor Ising model at the inverse temperature $\beta$.
However, our stationary measure $\mu_N$ is neither Bernoulli nor Gibbs, except $\gamma=0$ \cite{MR1400378}.

An elementary calculation shows
\begin{align*}
B(\rho)&=(1-\rho)\left\{1-2\gamma(1-2\rho)+\gamma^2(1-2\rho)^2 \right\},\\
D(\rho)&=\rho\left\{1+2\gamma(1-2\rho)+\gamma^2(1-2\rho)^2 \right\},
\end{align*}
and
\begin{align*}
F(\rho)&=-2(\rho-1/2)\left\{ 1-2\gamma + 4\gamma^2(\rho-1/2)^2 \right\},
\end{align*}
for each $\rho\in[0,1]$, and $V$ defined by
\begin{align*}
V(\rho)=(1-2\gamma)(\rho-1/2)^2 + 2\gamma^2(\rho-1/2)^4
\end{align*}
is a potential. $V$ has two local minima if, and only if,
$\gamma>1/2$, otherwise, a unique minimum.
\end{example}

Let us recall the notion of the {\it total variation mixing time}.
For any probability measures $\mu, \nu$ on $X_N$, define
\begin{align*}
\|\mu-\nu\|_{\TV} =  \max_{A\subset X_N}|\mu(A)-\nu(A)|
 = \dfrac{1}{2}\sum_{\eta\in X_N}|\mu(\eta)-\nu(\eta)| .
\end{align*}
Then, for each $0<\e<1$, we define the mixing time $t_{\mix}^N(\e)$ by
\begin{align*}
t^N_{\mix}(\e) =  \inf \left\{ t\ge0: \max_{\eta\in X_N} \|{{\bb P}}_\eta(\eta_t^N\in\cdot) - \mu_N(\cdot)\|_{\TV}\le\e\right\} .
\end{align*}

In this paper, we first study the mixing time of the reaction--diffusion model
when the potential $V$ has more than two local minima.
In this setting, we show that the mixing time has an exponential lower bound in $N$.
The precise statement is as follows.

\begin{theorem}\label{thm2-1}
Assume that the potential $V$ has $\ell$ local minima with $\ell\ge2$. 
Let $h_0$ be the constant given in \eqref{2-12}.
Then, for any $0<\e<\ell^{-1}$ and  any $N$ sufficiently large, we have
\begin{align*}
t_{\mix}^N(\e)  \ge  e^{Nh_0} .
\end{align*}
\end{theorem}

\begin{remark}\label{rem2-2}
The jump rate provided in Example \ref{exa2-1} is {\it attractive} in the sense that, for any configurations
$\eta, \xi$ such that $\eta(x)\ge\xi(x)$ for any $x$, it holds that
\begin{align*}
\begin{cases}
c(\eta) \le c(\xi), \quad \text{if $\eta(0)=\xi(0)=1$},\\
c(\eta) \ge c(\xi), \quad \text{if $\eta(0)=\xi(0)=0$}.
\end{cases}
\end{align*}
Note that $V$ is strictly convex if, and only if, $0\le \gamma < 1/2$.
In \cite{tanaka2020glauberexclusion}, we have established that for any attractive reaction--diffusion model
with a strictly convex potential, the mixing time is in the order of $\log N$.
Therefore, the reaction--diffusion model exhibits a phase transition regarding
the mixing time with the critical parameter $\gamma=1/2$.
\end{remark}

\begin{remark}\label{rem2-3}
It is natural to expect that
\begin{align*}
\lim_{N\to\infty}\dfrac1N \log t_{\mix}^N(\e)  =  \tilde h_0,
\end{align*}
for some constant $\tilde h_0$. To capture this asymptotic behavior,
the transition from a metastable well to another metastable well, so-called metastability, must be studied.
The metastable behavior of this model has been longstanding as an open problem (\cite[Chapter 10]{MR1707314}).
We leave this problem as future work.
\end{remark}

We also study the hitting times of rare events by applying large deviation estimates and some mixing time estimates
in the case where the potential has a unique minimum. Roughly speaking, we show that
when the process starts from a small neighborhood of the unique minimum of the potential,
the normalized hitting time of an open set that does not contain a unique minimum of the potential
converges to a mean one exponential random variable.
Since we need some notations to state this result, its precise statement is postponed to Theorem \ref{thm4-1}.

In Sections \ref{sec2} and \ref{sec3}, we assume that the potential $V$ has $\ell$ local minima with $\ell\ge2$,
whereas, in Section \ref{sec4}, we assume that the potential $V$ has a unique minimum.

The proof of Theorem \ref{thm2-1} mainly consists of Lemmata \ref{lem2-2} and \ref{lem2-3}.
Lemma \ref{lem2-2} provides some concentration results for the stationary states
established in \cite{MR3765880}.
Lemma \ref{lem2-3} provides an asymptotic estimate of the escape time from
a small neighborhood of metastable states.
We conclude this section by stating these lemmata and proving Theorem \ref{thm2-1}.

Let $\T$ be the one-dimensional continuous torus $\T=\R/\Z=[0,1)$
and $\mc M_{+}=\mc M_{+}(\T)$ be the space of all nonnegative measures on $\T$
with the total mass bounded by $1$, endowed with the weak topology. 
Note that $\mc M_{+}$ is compact under the weak topology.
For a
measure $\varrho$ in $\mc{M}_+$ and a continuous function
$G:\T\to\R$, denote the integral of $G$ with respect to $\varrho$ by $\lan \varrho, G \ran$
\begin{equation*}
\lan \varrho, G \ran  =  \int_\T G(\theta) \varrho(d\theta)  .
\end{equation*}
For a measurable function $\rho:\T\to\R$, let $\|\rho\|_2$ denote the $L^2$-norm with respect to the Lebesgue measure on $\T$
\begin{align*}
\|\rho\|_2^2=\int_\T\rho(\theta)^2d\theta.
\end{align*}
We also denote by $\lan\rho_1, \rho_2\ran$ the $L^2$-inner product
for measurable functions $\rho_1, \rho_2:\T\to[0,1]$
\begin{align*}
\lan\rho_1, \rho_2\ran=\int_\T\rho_1(\theta)\rho_2(\theta)d\theta.
\end{align*}

The space $\mc M_+$ is metrizable. By letting $e_{0}(\theta)=1$, $e_{k}(\theta)=
\sqrt 2\cos(2\pi k \theta),$ and $e_{-k}(\theta)= \sqrt 2\sin(2\pi k\theta)$,
$k\in\N$, one can define the distance $d$ on $\mc M_{+}$ using
\begin{equation}\label{2-9}
d(\varrho_{1},\varrho_{2})  =  \sum_{k\in\Z} \dfrac{1}{2^{|k|}}
\, |\lan\varrho_{1},e_{k}\ran - \lan\varrho_{2}, e_{k}\ran|   ,
\end{equation}
and one can show that the topology induced by this distance
corresponds to the weak topology.
Note that for any measurable functions $\rho,\rho':\T\to[0,1]$, we have
\begin{align}\label{2-17}
d(\rho, \rho') \le  3\|\rho -\rho'\|_2.
\end{align}

For each $\varrho\in\mc M_+$ and each $\alpha>0$, let $\mc B(\alpha;\varrho)$, $\mc B[\alpha;\varrho]$
be the $\alpha$-open, $\alpha$-closed neighborhood of $\varrho$ in $\mc M_+$,
respectively. When $\varrho(d\theta)=\rho(\theta)d\theta$ for a measurable function $\rho:\T\to[0,1]$,
we sometimes write these notions with $\rho$ instead of $\varrho$.
For instance, we denote $d(\varrho_1, \varrho_2)$ as $d(\rho_1, \rho_2)$
for $\varrho_i(d\theta)=\rho_i(\theta)d\theta, i=1,2.$

Let $\pi_N: X_N\to\mc M_+$ be the empirical measure defined by
\begin{align*}
\pi_N(\eta) = \frac{1}{N}\sum_{x\in \T_N}\eta(x)\delta_{x/N} , \quad \eta\in X_N,
\end{align*}
where $\delta_\theta$ is the Dirac measure that has a point mass at
$\theta\in\T$.
We also let $\mc P_N=\mu_N\circ(\pi_N)^{-1}$, which is a probability measure on
$\mc M_+$.

Let $S$ be the set of all classical solutions to the  semi-linear elliptic equation
\begin{align}\label{seeq}
(1/2)\Delta\rho + F(\rho)=0, \quad \text{on $\T$},
\end{align}
and $\mc M_\sol$ be the set of all measures whose density is a classical solution to \eqref{seeq}
\begin{align*}
\mc M_\sol=\{\bar\varrho\in \mc M_+: \bar\varrho(d\theta)=\bar\rho(\theta)d\theta, \bar\rho\in S\}.
\end{align*}

The following result has been established in \cite{MR3765880}.

\begin{lemma}\cite[Theorem 2.2]{MR3765880}\label{lem2-2}
For any $\alpha>0$, we have
\begin{equation*}
\lim_{N\to\infty}\mc P_N \left(\varrho\in\mc M_+: \min_{\bar\varrho\in\mc M_{\sol}} d\left(\varrho, \bar\varrho\right)\ge\alpha\right)  =  0 .
\end{equation*}
\end{lemma}

To describe a metastable well,
we recall the dynamical large deviation principle from \cite{MR1245249,MR3765880}.
Let $\mc M_{+,1}$ be the closed subset of $\mc M_{+}$ consisting of all
absolutely continuous measures with a density bounded by $1$
\begin{equation*}
\mc{M}_{+,1} =  \left\{\varrho\in\mc{M}_+:
\varrho(d\theta)=\rho(\theta)d\theta,\ 0\le\rho(\theta)\le1\ {\rm a.e.} \ \theta\in\T \right\} .
\end{equation*}
Fix $T>0$. Denote by $C^{m,n}([0,T]\times\T)$, $m,n$ in
$\N_{0}$, the set of all real functions defined on
$[0,T]\times\T$ which are $m$ times differentiable in the first
variable and $n$ times in the second one, and whose derivatives are continuous.

For each trajectory $\pi(t,d\theta)=\rho(t,\theta)d\theta$ in $D([0,T],\mc{M}_{+,1})$,
define the energy $\mc E_{T}(\pi)$ as
\begin{equation*}
\mc E_{T}(\pi) = \sup_{G\in C^{0,1}([0,T]\times\T)}
\left \{2\int_0^Tdt\ \lan\rho_t,\nabla G_t\ran
-\int_0^Tdt\int_{\T}d\theta\ G(t,\theta)^2 \right \} .
\end{equation*}
Note that the energy $\mc E_{T}(\pi)$ is finite if, and only if, $\rho$ has a
generalized derivative denoted by $\nabla\rho$, and this generalized
derivative is square--integrable on $[0,T]\times\T$
\begin{equation*}
\int_{0}^{T} dt \ \int_{\T} d\theta \ |\nabla\rho(t,\theta)|^{2} < \infty .
\end{equation*}
Here, we have
\begin{align*}
\mc E_T(\pi) =  \int_{0}^{T} dt \ \int_{\T} d\theta \ |\nabla\rho(t,\theta)|^{2} .
\end{align*}

For each function $G$ in $C^{1,2}([0,T]\times\T)$, define the
functional $\bar{J}_{T,G}:D([0,T],\mc{M}_{+,1})\to\R$ by
\begin{align*}
\bar{J}_{T,G}(\pi) &   = \lan\pi_T,G_T\ran -\lan\pi_{0},G_0\ran-
\int_0^Tdt\ \lan\pi_t, \partial_tG_t+\frac{1}{2}\Delta G_t\ran \\
&-\frac{1}{2}\int_0^Tdt\ \lan \chi(\rho_t), (\nabla G_t)^2\ran 
-\int_0^Tdt\ \left\{\lan B(\rho_t), e^{G_t}-1\ran + \lan D(\rho_t) , 
e^{-G_t}-1 \ran \right\},
\end{align*}
where $\chi(r)=r(1-r)$ is the mobility. 

Let $J_{T,G}: D([0,T],\mc{M}_+) \to[0,\infty]$ be the functional defined by
\begin{equation*}
J_{T,G} (\pi)  =  
\begin{cases}
\bar{J}_{T,G}(\pi)  & \text{if $\pi\in D([0,T],\mc{M}_{+,1})$} , \\
\infty  & \text{otherwise} ,
\end{cases}
\end{equation*}
and let $I_T:D([0,T],\mc{M}_+)\to[0,\infty]$ be the functional defined by
\begin{equation*}
I_T(\pi)  =  \begin{cases}
     \sup{J_{T,G}(\pi)} &
     \text{if $\mc E_{T}(\pi)<\infty$} , \\
     \infty  & \text{otherwise} ,
\end{cases}
\end{equation*}
where the supremum is carried over all functions $G$ in
$C^{1,2}([0,T]\times\T)$.
It has been established in \cite[Theorem 4.7]{MR3765880} that $I_T(\cdot)$
is lower semicontinuous and has compact level sets.

For any $T>0$ and any measurable function $\rho:\T\to[0,1]$,
define the dynamical large deviation rate function
$I_T(\cdot|\rho):D([0,T],\mc{M}_+)\to[0,\infty]$ as
\begin{equation*}
I_T(\pi|\rho)  =  \begin{cases}
     I_{T}(\pi) &
     \text{if $\pi(0,d\theta)=\rho(\theta)d\theta$} , \\
     \infty  & \text{otherwise} .
\end{cases}
\end{equation*}

We say that a sequence of initial configurations $\{\eta^N\}_N$
is {\it associated with} a measurable function $\rho:\T\to[0,1]$ if,
for any continuous function $G:\T\to\R$,
\begin{align*}
\lim_{N\to\infty} \dfrac{1}{N} \sum_{x\in\T_N} G(x/N)\eta^N(x)  =  \int_{\T} G(\theta)\rho(\theta)d\theta .
\end{align*}
Fix $T>0$ and let $Q_{\eta^N}$ be the distribution on the path space $D([0,T], \mc M_+)$
of the $\mc M_+$-valued process $\pi_N(\eta_\cdot^N)$ starting from a deterministic configuration $\pi_N(\eta^N)$.
The large deviation principle for the reaction--diffusion model was first
established in \cite[Theorem 2.2]{MR1245249} for the case where product measures provide the initial distribution.
When the process starts from a deterministic configuration, \cite{MR3765880} established the same result.

\begin{theorem}\cite[Theorem 2.5]{MR3765880}\label{thm2-2}
Assume that a sequence of initial configurations $\{\eta^N\}_N$
is associated with a measurable function $\rho:\T\to[0,1]$.
Then, for any closed set $\mf C \subset D([0,T], \mc M_+)$, we have
\begin{align*}
\limsup_{N\to\infty}\dfrac1N \log Q_{\eta^N}
\left( \mf C \right) \le - \inf_{\pi\in\mf C} I_T(\pi|\rho) .
\end{align*}
\end{theorem}

\begin{remark}\label{rem2-1}
Besides the assumption of Theorem \ref{thm2-2}, under the additional condition
\begin{itemize}
\item[\hypertarget{C}{(C)}] the functions $B$ and $D$ are concave,
\end{itemize}
the large deviation lower bound was also established in \cite{MR1245249,MR3765880}.
Namely, for any open set $\mf O \subset D([0,T], \mc M_+)$, we have
\begin{align*}
\liminf_{N\to\infty}\dfrac1N \log Q_{\eta^N}
\left( \mf O \right) \ge - \inf_{\pi\in\mf O} I_T(\pi|\rho) .
\end{align*}
Note that we do not use the large deviation lower bound in Sections \ref{sec2} and \ref{sec3},
so we do not have to assume the condition {\rm \hyperlink{C}{(C)}} in these sections.
Moreover, several results, which we cited in this paper from \cite{MR3765880, MR3947320}, remain in force
because they all hold without the condition {\rm \hyperlink{C}{(C)}} (the lower bound is not involved everywhere).
\end{remark}

Denote the positions at which the local minima of $V$ are attained by $\rho_1, \ldots, \rho_\ell$, $\ell\ge2$.
For each $i=1,\ldots,\ell$, let $\bar\rho_i$ be the constant function defined by $\bar\rho_i(\theta)\equiv\rho_i, \theta\in\T$
and $\bar\varrho_i$ the measure $\bar\varrho_i(d\theta)=\bar\rho_i(\theta)d\theta$.
Clearly, $\bar\varrho_i\in \mc M_\sol$ for each $i=1,\ldots,\ell$.

Fix $i=1,\ldots,\ell$. We construct small regions around $\bar\varrho_i$ denoted by $\mc A_i, \mc B_i$, and $\mc C_i$.
Let $\alpha_i,\beta_i$, and $\gamma_i$ be positive numbers chosen in the following way.
We denote the following sets by $\mc A_i, \mc B_i,$ and $\mc C_i$, respectively
\begin{align*}
\mc A_i = \mc B(\alpha_i;\bar\varrho_i) , \quad
\mc B_i =  \mc B[2\beta_i;\bar\varrho_i]\setminus \mc B(\beta_i;\bar\varrho_i)  , \quad
\mc C_i = \mc B(\gamma_i;\bar\varrho_i) .
\end{align*}

\begin{enumerate}
\item[\hypertarget{ga}{($\gamma$)}] First, choose $\gamma_i>0$ so that $\mc C_i$ does not intersect with 
the $\gamma_i$-closed neighborhood of $\mc M_\sol\setminus\{\bar\varrho_i\}$.
This choice is possible by Corollary \ref{cora-2}. Moreover, by Lemma \ref{lemb-4}, we have
\begin{align*}
2h_i:=\inf \{ V_i(\varrho): \varrho\notin \mc C_i \} > 0 ,
\end{align*}
where the function $V_i:\mc M_+\to[0,\infty]$ is the so-called {\it quasi-potential}
\begin{align*}
V_i(\varrho) =  \inf \left\{ I_T(\pi|\bar\rho_i): T>0, \pi\in D([0,T], \mc M_+), \pi_T=\varrho \right\} , \quad \varrho\in\mc M_+ .
\end{align*}
\item[($\beta$)] Second, we choose $\beta_i>0$ so that the following conditions are in force:
\begin{enumerate}
\item[($\beta$-1)] $2\beta_i<\gamma_i$.
\item[\hypertarget{be2}{($\beta$-2)}] For any $\rho(\theta)d\theta\in\mc B_i$, denote by $\rho_t$
the unique weak solution to the Cauchy problem
\begin{align}\label{hdleq}
\begin{cases}
\partial_t\rho =  (1/2)\Delta\rho +F(\rho) ,\\
\rho(0,\cdot) =  \rho(\cdot) .
\end{cases}
\end{align}
Then, it holds that
$\rho_t$ converges to $\bar\rho_i$ in the supremum norm as $t\to\infty$.
\item[\hypertarget{be3}{($\beta$-3)}] There exists $T_{1,i}>0$ such that
for any $\beta'\le\beta_i, T'\ge T_{1,i}$
and $\tilde\rho(\theta)d\theta\in\mc B[\beta';\bar\varrho_i]$,
\begin{align*}
\inf_{\pi\in \mf C_{T'}(\mc C_i^c)} I_{T'}(\pi|\tilde\rho)
 \ge  \inf \{ V_i(\varrho): \varrho\notin \mc C_i \}- h_i/2,
\end{align*}
where for each $T>0$ and each set $\mc D\subset\mc M_+$,
$\mf C_T(\mc D)$ stands for the subset of $D([0,T],
\mc M_{+})$ consisting of all trajectories $\pi$ for which there exists some
time $t \in [0, T]$ such that $\pi(t)$ belongs to $\mc D$ or
$\pi(t-)$ belongs to $\mc D$.
\end{enumerate}
\hyperlink{be2}{($\beta$-2)} and \hyperlink{be3}{($\beta$-3)} can be in force
using Lemmata \ref{lema-3} and \ref{lem3-3}, respectively.
\item[($\alpha$)] Finally, we choose $\alpha_i>0$ so that $\alpha_i<\beta_i$.
\end{enumerate}
We also set
\begin{align}\label{2-12}
h_0 =  \min\{h_i:i=1,\ldots,\ell\} .
\end{align}

For a general set $\mc D\subset \mc M_+$, we define the set
$\mc D^N\subset X_N$ by $\mc D^N=\pi_N^{-1}(\mc D)$.\footnote{We denote a subset of $\mc M_+$ by a calligraphic letter.
The corresponding subset with superscript $N$ is a subset of $X_N$.
For instance, for $\mc D\subset\mc M_+$, $\mc D^N$ is the subset of $X^N$ defined as above.}
Let $\mc A$ and $\mc C$ be the open sets given by
\begin{align*}
\mc A =  \bigcup_{i=1}^\ell \mc A_i , \quad \mc C =  \bigcup_{i=1}^\ell \mc C_i ,
\end{align*}
respectively, and $H_N$ be the hitting time of $[\mc C^N]^c$:
\begin{align}\label{2-10}
H_N =  \inf \left\{ t\ge0: \eta_t^N\in [\mc C^N]^c  \right\}.
\end{align}
In the previous display, and in what follows, for a general set $A$, $A^c$ denotes
the complement of $A$.
The following lemma is critical in proving Theorem \ref{thm2-1}.
Its proof is postponed to Section \ref{sec3}.

\begin{lemma}\label{lem2-3}
We have
\begin{align*}
\lim_{N\to\infty} \max_{\eta^N\in \mc A^N}\bb P_{\eta^N}\left( H_N \le e^{Nh_0} \right)  =  0 ,
\end{align*}
where $h_0$ is the constant defined in \eqref{2-12}.
\end{lemma}

We can now prove Theorem \ref{thm2-1}.

\begin{proof}[Proof of Theorem \ref{thm2-1}]

Fix $0<\e<\ell^{-1}$ and $\e_0=(1-\ell\e)/3$. We also let $\alpha_0=\min_i \alpha_i$.
By applying Lemma \ref{lem2-2} for $\alpha=\alpha_0$, 
there exists $N_0\ge1$ such that for all $N\ge N_0$, we have
\begin{equation}\label{2-2}
\mc P_N \left(\varrho\in\mc M_+: \min_{\bar\varrho\in\mc M_{\sol}} d\left(\varrho, \bar\varrho\right)\le\alpha_0\right)
 \ge  1 - \e_0  .
\end{equation}
For each $i=1,\ldots, \ell$, let $\tilde{\mc M}_i[\alpha_0]$ be the $\alpha_0$-closed neighborhood of $\tilde{\mc M}_i=\mc M_\sol\setminus\{\bar\varrho_i\}$. Note that
\begin{align*}
\left\{\varrho\in\mc M_+: \min_{\bar\varrho\in\mc M_{\sol}} d\left(\varrho, \bar\varrho\right)\le\alpha_0\right\} = \bigcup_{i=1}^\ell \tilde{\mc M}_i[\alpha_0].
\end{align*}
Therefore, the union bound shows that
\begin{align}\label{2-15}
\mc P_N \left(\varrho\in\mc M_+: \min_{\bar\varrho\in\mc M_{\sol}} d\left(\varrho, \bar\varrho\right)\le\alpha_0\right) \le \sum_{i=1}^\ell \mc P_N\left( \tilde{\mc M}_i[\alpha_0] \right).
\end{align}
It follows from \eqref{2-2} and \eqref{2-15} that for each $N\ge N_0$ there exists an integer $i_N\in\{1,\ldots,\ell\}$ such that
\begin{equation}\label{2-14}
\mc P_N \left(\tilde{\mc M}_{i_N}[\alpha_0] \right) \ge  (1 - \e_0)/\ell  .
\end{equation}

Let $j_N=i_N+1$ if $i_N=1,\ldots,\ell-1$, otherwise $j_N=1$ .
Let $\{\eta^N\}_N$ be a sequence in $\mc A_{j_N}^N$.
Then, from Lemma \ref{lem2-3}, there exists $N_1\ge1$ such that for all $N\ge N_1$, we have
\begin{align}\label{2-3}
\bb P_{\eta^N}\left( H_N \le e^{Nh_0} \right)  \le  \e_0/\ell .
\end{align}

We claim that for any $N\ge\max(N_0,N_1)$ and $t\le e^{Nh_0}$,
we have
\begin{align}\label{2-5}
\max_{\eta\in X_N} \|{{\bb P}}_\eta(\eta_t^N\in\cdot) - \mu_N(\cdot)\|_{\TV}  >  \e .
\end{align}
Once \eqref{2-5} is proven,
then the conclusion of Theorem \ref{thm2-1} is immediate by the definition of $t_\mix^N(\e)$.

Let us prove \eqref{2-5}. Let $N\ge\max(N_0,N_1)$ and $t\le e^{Nh_0}$. We have
\begin{align}\label{2-13}
\max_{\eta\in X_N} \|{{\bb P}}_\eta(\eta_t^N\in\cdot) - \mu_N(\cdot)\|_{\TV}
 &\ge  \|{{\bb P}}_{\eta^N}(\eta_t^N\in\cdot) - \mu_N(\cdot)\|_{\TV}\notag\\
 &\ge  \left| {{\bb P}}_{\eta^N}\left(\eta_t^N \in [\mc C_{i_N}^N]^c\right) - \mu_N([\mc C_{i_N}^N]^c)\right | \notag\\
 &\ge  \mu_N([\mc C_{i_N}^N]^c) -  {{\bb P}}_{\eta^N}\left(\eta_t^N\in [\mc C_{i_N}^N]^c\right)  .
\end{align}
Since $[\mc C_{i_N}^N]^c$ contains the set $\pi_N^{-1}(\tilde{\mc M}_{i_N}[\alpha_0])$,
we have from \eqref{2-14} and \eqref{2-13}
\begin{align}\label{2-7}
\max_{\eta\in X_N} \|{{\bb P}}_\eta(\eta_t^N\in\cdot) - \mu_N(\cdot)\|_{\TV}
&\ge  \mu_N \left( \pi_N^{-1} (\tilde{\mc M}_{i_N}[\alpha_0])\right)  - {{\bb P}}_{\eta^N}\left(\eta_t^N\in [\mc C_{i_N}^N]^c\right) \notag\\
&=  \mc P_N\left(\tilde{\mc M}_{i_N}[\alpha_0]\right) - {{\bb P}}_{\eta^N}\left(\eta_t^N\in [\mc C_{i_N}^N]^c\right) \notag\\
&\ge  (1 - \e_0)/\ell - {{\bb P}}_{\eta^N}\left(\eta_t^N\in [\mc C_{i_N}^N]^c\right).
\end{align}

Note that the event $\{\eta_t^N\in [\mc C_{i_N}^N]^c, H_N>e^{Nh_0}\}$ is empty
because $\eta_0^N=\eta^N\in \mc A_{j_N}^N$ and $t\le e^{Nh_0}$.
Therefore, we have from \eqref{2-3}
\begin{align}\label{2-8}
{{\bb P}}_{\eta^N}\left(\eta_t^N\in [\mc C_{i_N}^N]^c\right)
 &=  {{\bb P}}_{\eta^N}\left(\eta_t^N\in [\mc C_{i_N}^N]^c, H_N\le e^{Nh_0}\right)\notag\\
 &\le  {{\bb P}}_{\eta^N}\left(H_N\le e^{Nh_0}\right)  \le  \e_0/\ell .
\end{align}
Note that by the choice of $\e_0$ and $\e<\ell^{-1}$, we have
\begin{align}\label{2-6}
(1-2\e_0)/\ell = (1/3\ell)+(2\e/3) >  \e .
\end{align}
Therefore, \eqref{2-5} follows from \eqref{2-7},
\eqref{2-8}, and \eqref{2-6}, which completes proving Theorem \ref{thm2-1}.
\end{proof}

\section{Proof of Lemma \ref{lem2-3}}\label{sec3}

In this section, we prove Lemma \ref{lem2-3}.
Fix $i=1,\ldots,\ell$. We first observe that, starting from a configuration $\eta^N$ belonging to $\mc A_i^N$,
the process reaches $\mc A_j^N, j\neq i$ after exiting from $\mc C_i^N$.
Therefore, if we can show 
\begin{align}\label{3-14}
\lim_{N\to\infty} \max_{\eta^N\in\mc A_i^N}\bb P_{\eta^N}\left( H_N \le e^{Nh_i} \right)  =  0 ,
\end{align}
Lemma \ref{lem2-3} immediately follows from the definition of $h_0$ and \eqref{3-14}.

In what follows, we fix $i=1,\ldots,\ell$ and prove \eqref{3-14}.
Since $i$ is kept fixed, we sometimes omit dependence on $i$ for some notation.
Note that when the process starts from $\mc A_i^N$,
\begin{align}\label{3-21}
H_N =  \inf \left\{ t\ge0: \eta_t^N\in [\mc C_i^N]^c  \right\} .
\end{align}

Recall the definitions of $\mc A_i^N, \mc B_i^N$, and  $\mc C_i^N$ from Section \ref{sec2}.
We inductively define the sequence of stopping times, denoted by
$\tau_0\le\sigma_0\le\tau_1\le\sigma_2\le\tau_2\le\cdots$, as
$\tau_0=0$,
\begin{align*}
\sigma_k =  \inf \{t>\tau_k: \eta_t^N\in\mc B_i^N\}\quad \text{and} \quad \tau_k = \inf \{t>\sigma_{k-1}: \eta_t^N\in \mc A_i^N\cup [\mc C_i^N]^c \} .
\end{align*}
We avoid heavy notation by omitting dependence on $N$ and $i$ for $\sigma_k$ and $\tau_k$.
Note that we consider the process starting from $\mc A_i^N$ and only up to exiting from $\mc C_i^N$.
We also consider the discrete-time Markov chain $Z_k^N$ defined by $Z_k^N=\eta_{\tau_k}^N$.
Note that $Z_k^N$ is a Markov chain on $\mc A_i^N\cup [\mc C_i^N]^c$.
Let $\nu_N$ be the exit time of $Z_k^N$ from $\mc C_i^N$
\begin{align*}
\nu_N =  \inf\{k\in\N: Z_k^N\in [\mc C_i^N]^c\} .
\end{align*}

First, we estimate the one-step transition probability of $Z_k^N$
from $\mc A_i^N$ to $[\mc C_i^N]^c$. A similar estimate was given in 
\cite[Lemma 24]{MR3947320}. We emphasize that no lower bounds of the dynamical large deviation principle
are needed in what follows.

\begin{lemma}\label{lem3-1}
There exists $N_2$, such that for any $N\ge N_2$ and any sequence $\eta^N\in \mc A_i^N$, we have
\begin{align*}
\bb P_{\eta^N}\left( \nu_N=1 \right)  \le  e^{-Nh_i} .
\end{align*}
In particular,
\begin{align*}
\max_{\eta^N\in\mc A_i^N}\bb P_{\eta^N}\left( \nu_N=1 \right)  \le  e^{-Nh_i} .
\end{align*}

\end{lemma}

\begin{proof}
Recall the formula \eqref{3-21} and fix any sequence $\eta^N\in \mc A_i^N$.
When the process starts from $\eta^N$,
\begin{align*}
\left\{\nu_N=1\right\} = \left\{\eta^N_{\sigma_0}\in \mc B_i^N, \eta^N_{\tau_1}\in [\mc C_i^N]^c\right\} .
\end{align*}
Therefore, from the strong Markov property, we have
\begin{align*}
\bb P_{\eta^N}\left( \nu_N=1 \right)  &=  \bb P_{\eta^N}\left( \eta^N_{\sigma_0}\in \mc B_i^N, \eta^N_{\tau_1}\in [\mc C_i^N]^c \right) \\
 &=  \bb E_{\eta^N}\left[ {\bf 1} \left\{ \eta^N_{\sigma_0}\in \mc B_i^N \right\}\bb P_{\eta^N_{\sigma_0}}\left( \tau_1=H_N \right)  \right] \\
 &\le  \sup_{\xi\in\mc B_i^N} \bb P_\xi \left( \tau_1=H_N \right)  .
\end{align*}
Let $\{\xi_0^N\}_N$ be a sequence of $\mc B_i^N$ satisfying
\begin{align*}
\bb P_{\xi_0^N} \left(\tau_1=H_N \right)  =  \sup_{\xi\in\mc B_i^N} \bb P_\xi \left(\tau_1=H_N \right) .
\end{align*}
Let
\begin{align*}
h_*=-\limsup_{N\to\infty}\dfrac1N\log \bb P_{\xi_0^N} \left(\tau_1=H_N \right).
\end{align*}
By the definition of $h_*$, once we  show $h_*>h_i$,
there exists  some $N_2>0$ such that for any $N\ge N_2$ and any sequence $\eta^N\in \mc A_i^N$,
we have
\begin{align*}
\bb P_{\eta^N}\left( \nu_N=1 \right) \le  \bb P_{\xi_0^N} \left(\tau_1=H_N \right) \le e^{-Nh_i},
\end{align*}
and the conclusion follows.

To see $h_*>h_i$, we decompose the probability $\bb P_{\xi_0^N} \left( \tau_1=H_N \right)$ into
\begin{align*}
\bb P_{\xi_0^N}\left( \tau_1=H_N, H_N > T \right)
+ \bb P_{\xi_0^N}\left( \tau_1=H_N, H_N \le T \right)
\end{align*}
for each $T>0$. Then, we have
\begin{align}\label{3-16}
-h_*&  = \limsup_{N\to\infty}\dfrac1N\log \bb P_{\xi_0^N} \left(\tau_1=H_N \right) \notag\\
& \le \max \left\{\limsup_{N\to\infty}\dfrac1N\log\bb P_{\xi_0^N}\left( \tau_1=H_N, H_N > T \right),
\limsup_{N\to\infty}\dfrac1N\log \bb P_{\xi_0^N}\left( H_N \le T \right)\right\}
\end{align}
for any $T>0$.

The probability 
\begin{align*}
\dfrac1N\log\bb P_{\xi_0^N}\left( \tau_1=H_N, H_N > T \right)
\end{align*}
can be handled by Lemma \ref{lemb-2}. To see this, for each $\alpha>0$, let $\mc M_\sol(\alpha)$
be the $\alpha$-open neighborhood of $\mc M_\sol$
and $\tilde H_N(\alpha)$ be the hitting time of $[\mc M^N_\sol(\alpha)]^c$
\begin{align*}
\tilde H_N(\alpha) =  \inf \left\{ t\ge0: \eta_t^N\in [\mc M^N_\sol(\alpha)]^c  \right\} .
\end{align*}
On the event $\{\tau_1=H_N, H_N > T\}$,
the process starting from $\mc B_i^N$ does not hit $\mc M^N_{\sol}(\alpha_i)$ during $[0,T]$
because $\mc C_i$ does not intersect with the $\gamma_i$-closed neighborhood of $[\mc M_\sol\setminus\{\bar\varrho_i\}]$
by the conditions \hyperlink{ga}{($\gamma$)} and $\alpha_i<\gamma_i$.
Therefore, on the event $\{\tau_1=H_N, H_N > T\}$, we have
\begin{align*}
\tilde H_N(\alpha_i) \ge H_N \ge T.
\end{align*}
Moreover, by Lemma \ref{lemb-2}, there exist constants $T_0,C_0,$ and $N_0$, depending only on $\alpha_i$, such that
for all $N\ge N_0$ and all $k\ge 1$,
\begin{equation*}
\sup_{\eta\in X_N} \bb P_{\eta}\left[\tilde H_N(\alpha_i) \ge kT_0\right]
\le e^{ -k C_0 N}.
\end{equation*}
By letting $k=(2h_i)/C_0, \tilde T_0=(2h_iT_0)/C_0$ in the previous display, we have
\begin{align}\label{3-12}
 \bb P_{\xi_0^N}\left( \tau_1=H_N, H_N > \tilde T_0 \right) \le  \sup_{\eta\in X_N} \bb P_{\eta}\left[\tilde H_N(\alpha_i) \ge \tilde T_0\right]\le e^{ -2h_iN}
\end{align}
for all $N\ge N_0$.

Let us turn to the probability
\begin{align*}
\bb P_{\xi_0^N}\left( H_N \le T \right).
\end{align*}
Because $\mc B_i$ is compact in $\mc M_+$, there exists a subsequence $\{\xi_0^{N_k}\}_k$ of $\{\xi_0^N\}$ such that
$\pi_N(\xi_0^{N_k})$ converges to some $\varrho\in\mc B_i$ as $k\to\infty$.
Moreover, $\varrho$ is absolutely continuous with respect to the Lebesgue measure on $\T$ because
each configuration in $X_N$ has at most one particle per site.
Let $\rho:\T\to[0,1]$ be the density of $\varrho$: $\varrho(d\theta)=\rho(\theta)d\theta$.
We can also assume the loss of generality that the sequence $\{\xi_0^{N_k}\}_k$ satisfies
\begin{align*}
\limsup_{N\to\infty}\dfrac1N\log \bb P_{\xi_0^N}\left( H_N \le T\right)
= \lim_{k\to\infty}\dfrac{1}{N_k}\log \bb P_{\xi_0^{N_k}}\left( H_{N_k} \le  T\right)
\end{align*}

For each $T>0$ and each set $\mc D\subset\mc M_+$, recall the definition of $\mf C_T(\mc D)$,
which was introduced in the condition \hyperlink{be2}{($\beta$-2)}.
Note that if $\mc D$ is a closed subset of $\mc M_+$,
$\mf C_T(\mc D)$ is a closed subset of $D([0,T], \mc M_+)$.
When the process starts from $\mc B_i^N$, we have
\begin{align*}
\{ H_N\le T\}  \subset  \{ \pi_N(\eta_\cdot^N) \in \mf C_T(\mc C_i^c) \} .
\end{align*}
Therefore, by Theorem \ref{thm2-2}, we have
\begin{align*}
\lim_{k\to\infty}\dfrac{1}{N_k}\log \bb P_{\xi_0^{N_k}}\left( H_{N_k} \le  T\right)
 &\le   \limsup_{k\to\infty}\dfrac{1}{N_k} \log {{\bb P}}_{\xi_0^{N_k}} \left( \pi_{N_k}(\eta_\cdot^{N_k}) \in \mf C_T(\mc C_i^c)  \right)\\
 &=   \limsup_{k\to\infty}\dfrac{1}{N_k} \log Q_{\xi_0^{N_k}}\left( \mf C_T(\mc C_i^c)  \right)\\
 &\le  - \inf_{\pi\in\mf C_T(\mc C_i^c)} I_T(\pi|\rho).
\end{align*}
Recall the definition of $T_{1,i}$ and the condition \hyperlink{be2}{($\beta$-2)}.
For $\tilde T_i=\max(\tilde T_0, T_{1,i})$, by the condition \hyperlink{be2}{($\beta$-2)} the last expression is bounded by
\begin{align*}
-\inf \{ V_i(\varrho): \varrho\notin \mc C_i \} + h_i/2 = -(3/2)h_i.
\end{align*}
Therefore, we have
\begin{align}\label{3-13}
\limsup_{N\to\infty}\dfrac1N \log {{\bb P}}_{\xi_0^N} \left( H_N \le \tilde T_i \right)
 \le  - (3/2)h_i.
\end{align}

Taking $T=\tilde T_0$ in \eqref{3-16}, by \eqref{3-12} and \eqref{3-13}, we have
\begin{align*}
-h_*  &\le  \max \left\{\limsup_{N\to\infty}\dfrac1N\log\bb P_{\xi_0^N}\left( \tau_1^N=H_N, H_N > \tilde T_0 \right),
\limsup_{N\to\infty}\dfrac1N\log \bb P_{\xi_0^N}\left( H_N \le \tilde T_0 \right)\right\}\\
& \le \max \left\{-2h_i, \limsup_{N\to\infty}\dfrac1N\log \bb P_{\xi_0^N}\left( H_N \le \tilde T_i \right)\right\}\\
& \le \max \left\{-2h_i, -(3/2)h_i\right\} < -h_i,
\end{align*}
concluding the proof.
\end{proof}

The proof of the following lemma is close to that of \cite[Lemma 23]{MR3947320}.

\begin{lemma}\label{lem3-3}
Fix $i=1,\ldots, \ell$ and any $\e>0$. There exist $\beta_i=\beta_i(\e)>0$ and $T_{1,i}=T_{1,i}(\e)>0$ such that
for any $\beta'\le\beta_i, T'\ge T_{1,i}$ and  $\tilde\rho(\theta)d\theta\in\mc B[\beta';\bar\varrho_i]$,
\begin{align}\label{3-11}
\inf_{\pi\in \mf C_{T'}(\mc C_i^c)} I_{T'}(\pi|\tilde\rho)
 \ge  \inf \{ V_i(\varrho): \varrho\notin \mc C_i \} - \e  .
\end{align}
\end{lemma}

\begin{proof}
Assume that the conclusion of the lemma fails.
Here, for any $\beta>0$ and $T>0$,
there exist $\beta' \le \beta, T'\ge T, \tilde\rho(\theta)d\theta\in\mc B[\beta';\bar\varrho_i]$
and $\pi(t,d\theta) \in \mf C_{T'}(\mc C_i^c)$ such that
\begin{align*}
I_{T'}(\pi|\tilde\rho)
 <  \inf \{ V_i(\varrho): \varrho\notin \mc C_i \} - \e .
\end{align*}
In particular, by letting $\beta=1/n$ and $T=1$ for any $n\in\N$,
there exist $\beta'_n \le 1/n, T_n'\ge1, \tilde\rho_n(\theta)d\theta\in\mc B[\beta'_n;\bar\varrho_i]$
and $\pi_n \in \mf C_{T_n'}(\mc C_i^c)$ such that
\begin{align}\label{3-18}
I_{T_n'}(\pi_n|\tilde\rho_n)
 <  \inf \{ V_i(\varrho): \varrho\notin \mc C_i \} - \e .
\end{align}
By Lemma \ref{lemb-5}, $\pi_n(t,d\theta)$ has a density $\rho_n(t,\theta)$ for each $t\in[0,T_n']$ and the trajectory $\pi_n\mapsto\pi_n(t,d\theta)$ is continuous in $\mc M_+$.
By the definition of $\mf C_{T'_n}(\mc C_i^c)$ and the continuity of $\pi_n$,
there exists $0<T_n''\le T_n'$ such that $\pi_n(T_n'')\in \mc C_i^c$.

Let $\bb D$ be the space of measurable functions $\rho:\T\to[0,1]$ endowed with the $L^2$-topology
and define the function $\bb V_i:\bb D\to\R$ by
\begin{align*}
\bb V_i(\rho)=V_i(\rho(\theta)d\theta), \quad \rho\in\bb D.
\end{align*}
Note that $\bb V_i(\bar\rho_i)=0$ and, by \cite[Theorem 6]{MR3947320}, $\bb V_i$ is continuous at $\bar\rho_i$ in $\bb D$.
Therefore, there exists $\beta_\dag>0$ such that $6\beta_\dag<\gamma_i$ and
$\bb V_i(\rho)\le\e/2$ for any $\rho\in\bb D$ satisfying $\|\rho - \bar\rho_i\|_2\le 2\beta_\dag$.
Note that the set $\{ \rho(\theta)d\theta:\rho\in \bb D, \|\rho - \bar\rho_i\|_2\le 2\beta_\dag\}$ is a closed subset of $\mc M_+$.

Define $T_n'''\ge0$ by
\begin{align*}
T_n'''=\sup \{0\le t \le T_n'': \| \rho_n(t,\cdot) - \bar\rho_i(\cdot)\|_2\le 2\beta_\dag\}.
\end{align*}
When the set inside the supremum is empty, let $T_n'''=0$.
Note that $T_n'''<T_n''$ for any $n\in\N$.
To see this, we can assume that $T_n'''>0$ and take a sequence $0<t_k\uparrow T_n''$ such that
\begin{align*}
\| \rho_n(t_k,\cdot) - \bar\rho_i(\cdot)\|_2\le 2\beta_\dag
\end{align*}
for any $k\in\N$.
Because this condition is closed under the weak topology, letting $k\to\infty$ provides
\begin{align}\label{3-20}
\| \rho_n(T_n'',\cdot) - \bar\rho_i(\cdot)\|_2\le 2\beta_\dag
\end{align}
By \eqref{2-17} and $6\beta_\dag<\gamma_i$, we have
\begin{align*}
d(\rho_n(T_n''), \bar\rho_i) \le3 \| \rho_n(T_n'',\cdot) - \bar\rho_i(\cdot)\|_2\le 6\beta_\dag <\gamma_i.
\end{align*}
Therefore, the fact $\pi_n(T_n'')\in \mc C_i^c$ yields $T_n'''<T_n''$.
Let $\tilde T_n=T_n''-T_n'''$ and define the trajectory $\tilde\pi_n(t,d\theta)=\tilde\rho_n(t,\theta)d\theta \in D([0,\tilde T_n],\mc M_+)$ by
$\tilde\pi_n(t,d\theta)=\pi_n(t+T_n''', d\theta)$ for $t\in[0,\tilde T_n]$.
 
For each $\beta>0$ and $T>0$, let $\mb D_{T,\beta}$ be the set of trajectories $\pi(t,d\theta)=\rho(t,\theta)d\theta\in D([0,T], \mc M_{+,1})$ such that $\|\rho(t,\cdot) -\bar\rho(\cdot)\|_2>\beta$ for all $0 \le t \le T$ and all $\bar\rho\in S$.
By Lemma \ref{lemb-3}, there exists $T_\dag$ such that
\begin{align}\label{3-17}
\inf_{\pi\in\mb D_{T_\dag, \beta_\dag}} I_T(\pi) \ge
\inf \{ V_i(\varrho): \varrho\notin \mc C_i \}.
\end{align}
On the other hand, by the construction of $\tilde\pi_n$, we have
\begin{align}\label{3-19}
I_{T_n'}(\pi_n)\ge I_{\tilde T_n}(\tilde\pi_n),
\end{align}
and $\tilde\pi_n\in \mb D_{\tilde T_n, \beta_\dag}$ for any $n\in\N$.

Let us consider the case that there exists $n$ such that $\tilde T_n\ge T_\dag$.
Here, by \eqref{3-17} we obtain
\begin{align*}
I_{\tilde T_n}(\tilde\pi_n) \ge I_{T_\dag}(\tilde\pi_n) \ge \inf \{ V_i(\varrho): \varrho\notin \mc C_i \},
\end{align*}
contradicting \eqref{3-18} and \eqref{3-19}.
Therefore, it remains to consider the case that $\tilde T_n \le T_\dag$ for any $n\in\N$.

Assume that $\tilde T_n \le T_\dag$ for any $n\in\N$.
Here, we extend $\tilde \pi_n$ as a trajectory in $D([0,T_\dag],\mc M_+)$ in the following way
\begin{align*}
\tilde \pi_n(t,d\theta)=
\begin{cases}
\tilde\rho(t,\theta)d\theta, \quad 0\le t \le \tilde T_n,\\
\bar\rho_n(t-\tilde T_n,\theta)d\theta, \quad \tilde T_n \le t \le T_\dag,
\end{cases}
\end{align*}
where $\bar\rho_n$ is the solution to the Cauchy problem \eqref{hdleq} with the initial condition $\tilde\rho_n(\tilde T_n)$.
By Lemma \ref{lemb-1}, we have
\begin{align*}
I_{T_\dag}(\tilde\pi_n) = I_{\tilde T_n}(\tilde\pi_n).
\end{align*}
Therefore, by \eqref{3-18} and \eqref{3-19}, we have
\begin{align*}
I_{T_\dag}(\tilde\pi_n) <  \inf \{ V_i(\varrho): \varrho\notin \mc C_i \} - \e.
\end{align*}

Since $I_{T_\dag}$ has compact level sets, there exists a subsequence of $\tilde\pi_n$ converging to some $\tilde\pi(t,d\theta)=\tilde\rho(t,\theta)d\theta\in D([0,T_\dag],\mc M_{+,1})$ such that
\begin{itemize}
\item[(i)] $\|\tilde\rho(0,\cdot)-\bar\rho_i(\cdot)\|_2 \le 2\beta_\dag$.
\item[(ii)] There exists some $0\le \tilde T \le T_\dag$ such that $\tilde \pi(\tilde T)\in \mc C_i^c$.
\end{itemize}
To see (i), we must consider cases $T_n'''>0$ and $T_n'''=0$. In the former case,
(i) follows from \eqref{3-20}. In the latter case, (i) follows from
$\tilde\pi_n(0)=\pi_n(T_n''')=\pi_n(0)\in\mc B[1/n;\bar\rho_i]$, thereby, $\tilde\rho(0)=\bar\rho_i$.
(ii) follows from $0\le \tilde T_n\le T_\dag$, $\tilde\pi_n(\tilde T_n)=\pi_n(T_n'')\in\mc C_i^c$ and
$\tilde \pi_n$ converges to $\tilde\pi$ in the uniform topology.
Moreover, because $I_{T_\dag}$ is lower semicontinuous, we have
\begin{align*}
I_{T_\dag}(\tilde \pi)\le \liminf_{n\to\infty}I_{T_\dag}(\tilde\pi_n)
<  \inf \{ V_i(\varrho): \varrho\notin \mc C_i \} - \e.
\end{align*}

By (i) and the choice of $\beta_\dag$,
there exist some $T^{(0)}>0$ and a trajectory $\pi^{(0)}\in D([0,T^{(0)}], \mc M_+)$ such that
$\pi^{(0)}(0)=\bar\varrho_i, \pi^{(0)}(T^{(0)})=\tilde\pi(0)$ and $I_{T^{(0)}}(\pi^{(0)})\le \e/2$.
Let $T=T^{(0)}+\tilde T$ and define the trajectory $\pi\in D([0,T], \mc M_+)$ by
\begin{align*}
\pi(t,d\theta)=
\begin{cases}
\pi^{(0)}(t,d\theta), \quad 0\le t \le T^{(0)},\\
\tilde \pi(t-T^{(0)},d\theta), \quad T^{(0)} \le t \le T.
\end{cases}
\end{align*}
Then we have
\begin{align*}
I_T(\pi) = I_{T^{(0)}}(\pi^{(0)})+I_{T_\dag}(\tilde\pi) < \inf \{ V_i(\varrho): \varrho\notin \mc C_i \} - \e/2.
\end{align*}

Finally, since the trajectory $\pi$ satisfies $\pi(0)=\pi^{(0)}(0)=\bar\varrho_i$ and $\pi(T)=\tilde\pi(\tilde T)\in \mc C_i^c$,
by the definition of $V_i$,  we have
\begin{align*}
I_T(\pi) \ge \inf \{ V_i(\varrho): \varrho\notin \mc C_i \}.
\end{align*}
This contradicts the penultimate display, which completes proving the lemma.
\end{proof}

Secondly, we show that the process does not escape quickly  from $\mc C_i^N$
with probability less than one half.
Moreover, this occurs uniformly in the starting points in $\mc A_i^N$.

\begin{lemma}\label{lem3-2}
There exists $T_2>0$, such that 
\begin{align*}
\lim_{N\to\infty} \max_{\eta^N\in \mc A_i^N}\bb P_{\eta^N}\left( \tau_1 \le  T_2 \right)  = 0.
\end{align*}
In particular, there exist $T_2>0$ and $N_3>0$ such that for any $N\ge N_3$ and  sequence $\eta^N\in \mc A_i^N$, we have
\begin{align*}
\bb P_{\eta^N}\left( \tau_1 >  T_2 \right)  \ge  1/2.
\end{align*}
\end{lemma}
\begin{proof}
Let $T_2>0$. The probability in the lemma can be decomposed into
\begin{align}\label{3-2}
\bb P_{\eta^N}\left( \tau_1 \le  T_2, \eta_{\tau_1}^N\in \mc A_i^N \right)
+ \bb P_{\eta^N}\left( \tau_1 \le  T_2, \eta_{\tau_1}^N\in [\mc C_i^N]^c \right) .
\end{align}
By the definition of $\nu_N$, $\bb P_{\eta^N}\left( \nu_N=1 \right)$ bounds
the second probability in \eqref{3-2}. From Lemma \ref{lem3-1},
this probability vanishes as $N\to\infty$.
However, by the strong Markov property 
the first probability in \eqref{3-2} equals
\begin{align}\label{3-3}
\bb E_{\eta^N}\left[ 
{{\bb P}}_{\eta_{\sigma_0}^N}\left( \tau_1 \le T_2, \eta_{\tau_1}^N\in \mc A_i^N \right) \right] .
\end{align}

Recall the definition of $\mf C_{T_2}(\mc B[\alpha_i;\bar\varrho_i])$,
which is introduced in the condition \hyperlink{be3}{($\beta$-3)}.
Note that \eqref{3-3} is bounded by
\begin{align*}
\max_{\xi^N\in\mc B_i^N}
\bb P_{\xi^N}\left(\eta_\cdot^N: \pi_N(\eta_\cdot^N)\in \mf C_{T_2}(\mc B[\alpha_i;\bar\varrho_i]) \right)
= \max_{\xi^N\in\mc B_i^N}
Q_{\xi^N}\left(\mf C_{T_2}(\mc B[\alpha_i;\bar\varrho_i]) \right).
\end{align*}
Let $\{\xi_0^N\}_N$ be a sequence satisfying
\begin{align*}
Q_{\xi_0^N}\left(\mf C_{T_2}(\mc B[\alpha_i;\bar\varrho_i]) \right)
 =  \max_{\xi^N\in\mc B_i^N}
Q_{\xi^N}\left(\mf C_{T_2}(\mc B[\alpha_i;\bar\varrho_i]) \right) .
\end{align*}
Performing an argument, as we did in the proof of Lemma \ref{lem3-1} (see the paragraph after \eqref{3-12}),
there exists a subsequence $\{\xi_0^{N_k}\}_k$ of $\{\xi_0^N\}_N$ such that
$\{\xi_0^{N_k}\}_k$ is associated with some $\rho:\T\to[0,1]$ with $\rho(\theta) d\theta\in\mc B_i$ and
\begin{align*}
\limsup_{N\to\infty} \dfrac1N \log Q_{\xi_0^N}\left(\mf C_{T_2}(\mc B[\alpha_i;\bar\varrho_i]) \right)
= \lim_{k\to\infty} \dfrac{1}{N_k} \log Q_{\xi_0^{N_k}}\left(\mf C_{T_2}(\mc B[\alpha_i;\bar\varrho_i]) \right).
\end{align*}
Then, by Theorem \ref{thm2-2}, the right-hand side of the last display is bounded above by
\begin{align*}
 - \inf_{\pi\in \mf C_{T_2}(\mc B[\alpha_i;\bar\varrho_i])} I_{T_2}(\pi|\rho) .
\end{align*}

It remains to show that there exists $T_2>0$ such that
\begin{align}\label{3-15}
\inf_{\pi\in \mf C_{T_2}(\mc B[\alpha_i;\bar\varrho_i])} I_{T_2}(\pi|\rho)>0.
\end{align}
To see this, for each $\tilde\rho(\theta)d\theta\in\mc B_i$,
let $\tilde\rho_t$ be the solution to the Cauchy problem \eqref{hdleq},
with the initial condition $\tilde\rho$.
Let $\tau^{(i)}(\tilde\rho)$ be the first entrance time of $\tilde\rho_t$ into $\mc B[\alpha_i;\bar\varrho_i]$, that is,
\begin{align*}
\tau^{(i)}(\tilde\rho)= \inf\left\{ t\ge0: \tilde\rho_t \in \mc B[\alpha_i;\bar\varrho_i] \right\}.
\end{align*}
Note that $\tau^{(i)}(\tilde\rho)$ is finite because by $(\beta$-2), $\tilde\rho_t$ converges to $\bar\rho_i$
as $t\to\infty$ for any $\tilde\rho(\theta)d\theta\in\mc B_i$. Moreover, from Corollary \ref{cora-1},
the application $\tilde\rho(\theta)d\theta\in\mc B_i\cap\mc M_{+,1}\mapsto\tau^{(i)}(\tilde\rho)$
is lower semicontinuous with respect to the weak topology.
Let $\tilde T_2$ be the constant defined by
\begin{align*}
\tilde T_2= \inf \left\{ \tau^{(i)}(\tilde\rho): \tilde\rho(\theta)d\theta\in\mc B_i \right\} = 
\min \left\{ \tau^{(i)}(\tilde\rho): \tilde\rho(\theta)d\theta\in\mc B_i \right\}.
\end{align*}
The second equality follows from the compactness of $\mc B_i$ and the mentioned lower semi-continuity of $\tau^{(i)}$.
Let $T_2=\tilde T_2/2$.

Before turning to show \eqref{3-15}, we see that  some $\tilde\pi\in \mf C_{T_2}(\mc B[\alpha_i;\bar\varrho_i])$
attains the infimum in \eqref{3-15}.
Indeed, let us take any $\pi' \in \mf C_{T_2}(\mc B[\alpha_i;\bar\varrho_i])$ satisfying $I_{T_2}(\pi'|\rho)<\infty$.
Because $I_{T_2}(\cdot|\rho)$ has a compact level set and $\mf C_{T_2}(\mc B[\alpha_i;\bar\varrho_i])$ is closed,
the subset of $D([0,T_2], \mc M_+)$ defined by
\begin{align*}
\mf C' := \left\{ \pi\in\mf C_{T_2}(\mc B[\alpha_i;\bar\varrho_i]): I_{T_2}(\pi|\rho)\le I_{T_2}(\pi'|\rho) \right\},
\end{align*}
is a compact subset of $D([0,T_2], \mc M_+)$. Because $I_{T_2}(\cdot|\rho)$ is lower semicontinuous,
there exists some $\tilde\pi\in \mf C_{T_2}(\mc B[\alpha_i;\bar\varrho_i])$ such that
\begin{align*}
 I_{T_2}(\tilde \pi|\rho)
= \inf_{\pi\in \mf C'} I_{T_2}(\pi|\rho) 
= \inf_{\pi\in \mf C_{T_2}(\mc B[\alpha_i;\bar\varrho_i])} I_{T_2}(\pi|\rho).
\end{align*}
Therefore, $\tilde\pi\in \mf C_{T_2}(\mc B[\alpha_i;\bar\varrho_i])$ can attain the infimum in \eqref{3-15}.

Now, assume that $I_{T_2}(\tilde \pi|\rho)=0$. Here, from Lemma \ref{lemb-1}, the density of $\tilde\pi(t,d\theta)$,
denoted by $\tilde\rho(t,\theta), 0\le t \le T_2$, is the weak solution to the Cauchy problem \eqref{hdleq}
with the initial condition $\rho$. However, this contradicts $\tilde\pi \in \mf C_{T_2}(\mc B[\alpha_i;\bar\varrho_i])$
and $\tilde\rho_t(\theta)d\theta\notin \mc B[\alpha_i;\bar\varrho_i]$ for any $0\le t \le T_2$.
Therefore, \eqref{3-15} has been shown, completing the proof of Lemma \ref{lem3-2}.
\end{proof}

Invoking Lemmata \ref{lem3-1} and \ref{lem3-2}, proving Lemma \ref{lem2-3}
is similar to the one of \cite[Chapter 4, Theorem 4.2]{MR1652127}.

\begin{proof}[Proof of Lemma \ref{lem2-3}]
As mentioned in the first paragraph of this section,
to prove Lemma \ref{lem2-3} it is enough to show \eqref{3-14}.
Fix any sequence $\eta^N\in \mc A_i^N$, such that
\begin{align*}
\bb P_{\eta^N}\left( H_N \le e^{Nh_i} \right) = \max_{\zeta^N\in\mc A_i^N} \bb P_{\zeta^N}\left( H_N \le e^{Nh_i} \right).
\end{align*}
First, note that
\begin{align*}
\bb P_{\eta^N}\left( H_N \le e^{Nh_i} \right)
 \le  \bb P_{\eta^N}\left( \tau_1= H_N \right) + \bb P_{\eta^N}\left( \tau_1 < H_N, H_N \le e^{Nh_i} \right)  .
\end{align*}
The first probability of the right-hand side equals
\begin{align*}
\bb P_{\eta^N}\left( \nu_N=1 \right) ,
\end{align*}
and vanishes as $N\to\infty$ by Lemma \ref{lem3-1}.
On the other hand, by the strong Markov property,
the last probability of the penultimate display is bounded by
\begin{align*}
\sum_{k=1}^\infty\bb E_{\eta^N}\left[ {\bf 1}\left\{ \tau_1 < H_N\right\} 
{{\bb P}}_{\eta_{\tau_1}^N}\left(\nu_N=k, H_N \le e^{Nh_i}\right) \right] .
\end{align*}

Note that, on the event $\{ \tau_1 < H_N\}$,
we have $\eta_{\tau_1}^N\in \mc A_i^N$.
Let $T_2$ be the constant chosen according to Lemma \ref{lem3-2}.
We also let $m_N=\lceil(3/T_2)e^{N^{1/2}h_i}\rceil$, where $\lceil \cdot \rceil$ denotes the ceil function.
For any configuration $\xi^N\in \mc A_i^N$, we have
\begin{align}\label{3-1}
&\sum_{k=1}^\infty {{\bb P}}_{\xi^N}\left(\nu_N=k, H_N \le e^{Nh_i}\right)\notag\\
& \quad  \le  {{\bb P}}_{\xi^N}\left(\nu_N\le m_N\right) + \sum_{k=m_N}^\infty {{\bb P}}_{\xi^N}\left(\nu_N=k, \tau_k \le e^{Nh_i}\right)\notag\\
& \quad  \le  {{\bb P}}_{\xi^N}\left(\nu_N\le m_N\right) + {{\bb P}}_{\xi^N}\left(\nu_N\ge m_N, \tau_{m_N}  \le e^{Nh_i}\right) .
\end{align}

From Lemma \ref{lem3-1},
\begin{align*}
\bb P_{\xi^N}\left( \nu_N=1 \right)  \le  e^{-Nh_i} ,
\end{align*}
for any $N$ sufficiently large and any $\xi^N\in \mc A_i^N$. Therefore, by the strong Markov property, we have
\begin{align*}
\bb P_{\xi^N}\left( \nu_N> m_N \right)  \ge  \left( 1 - e^{-Nh_i} \right)^{m_N}  .
\end{align*}
As the right-hand side of the last expression converges to $1$ as $N\to\infty$,  we have
\begin{align*}
\lim_{N\to\infty} \max_{\xi^N\in\mc A_i^N}{{\bb P}}_{\xi^N}\left(\nu_N\le m_N\right) =  0 .
\end{align*}

Let us address the second probability in \eqref{3-1}. 
From the trivial decomposition
\begin{align*}
\tau_{m_N}=\left( \tau_{m_N}-\tau_{m_N-1} \right) + \left( \tau_{m_N-1}-\tau_{m_N-2} \right) + \cdots + \left( \tau_{1}-\tau_{0} \right) ,
\end{align*}
$\tau_{m_N}$ can be bounded below by
\begin{align*}
T_2 \sum_{k=1}^{m_N}{\bf 1}\left\{\tau_k - \tau_{k-1}> T_2\right\} =:  R_N .
\end{align*}
Therefore,
\begin{align*}
{{\bb P}}_{\xi^N}\left(\nu_N\ge m_N, \tau_{m_N} \le e^{Nh_i}\right)
 &\le 
{{\bb P}}_{\xi^N}\left(\nu_N\ge m_N, R_N  \le e^{Nh_i}\right) \\
 &\le  {{\bb P}}_{\xi^N}\left(\nu_N\ge m_N, \dfrac{R_N}{T_2m_N} \le 1/3\right) .
\end{align*}
Note that $R_N/T_2$ is the sum of independent Bernoulli random variables.
Moreover, by Lemma \ref{lem3-2}, there exists $N_4$ such that,
for any $N\ge N_4$ and any configuration $\xi^N\in\mc A_i^N$,
under $\bb P_{\xi^N}$ and on the event $\{\nu_N \ge m_N\}$
the mean of each increment of $R_N/T_2$ is larger than $1/2$.
Therefore, the last probability vanishes as $N\to\infty$ uniformly in $\xi^N\in\mc A_i^N$. Therefore,
\begin{align*}
\lim_{N\to\infty} \max_{\xi^N\in\mc A_i^N}
{{\bb P}}_{\xi^N}\left(\nu_N\ge m_N, \tau_{m_N} \le e^{Nh_i}\right) =  0 ,
\end{align*}
completing the proof of \eqref{3-14} and, hence, Lemma \ref{lem2-3}.
\end{proof}

\section{Hitting times of rare events}\label{sec4}

In this section, we study the hitting times of rare events in the case where
\begin{enumerate}
\item[\hypertarget{UM}{(UM)}] the potential $V$ has a unique minimum.
\end{enumerate}
Denote by $\rho_*$ the position at which the minimum of $V$ is attained.
As before, let $\bar\rho_*(\theta) \equiv \rho_*, \theta\in\T$ and $\bar\varrho_*(d\theta)=\bar\rho_*(\theta)d\theta$, respectively.

Fix an open subset $\mc O$ of $\mc M_+$ such that
$\mc O\cap \mc M_{+,1}\neq\emptyset$ and
\begin{align*}
d(\bar\varrho_*, \mc O):=\inf_{\varrho\in \mc O}d(\bar\varrho_*, \varrho) >0.
\end{align*}
Note that
\begin{align*}
\mu_N(\mc O^N)
= \mc P_N \left(\mc O\right)
\le \mc P_N \left(\varrho\in\mc M_+: d\left(\bar\varrho_*, \varrho\right)\ge d(\bar\varrho_*, \mc O)\right).
\end{align*}
Under condition \hyperlink{UM}{(UM)}, 
the  semi-linear elliptic equation \eqref{seeq} admits a unique classical solution given by $\bar\rho_*$.
Therefore, by Lemma \ref{lem2-2}
the last expression vanishes as $N\to\infty$.
Thus, we have
\begin{align}\label{4-5}
\lim_{N\to\infty} \mu_N(\mc O^N) =0.
\end{align}

We apply our large deviation estimates and some mixing time estimates 
to show the convergence of hitting times of rare events.
Therefore, let $H_N^{\mc O}$ be the hitting time of $\mc O^N$
\begin{align*}
H_N^{\mc O}=\inf\{ t\ge0: \eta_t^N\in \mc O^N\}. 
\end{align*}
To establish the convergence of $H_N^{\mc O}$, we need the following result,
which is Lemma \ref{lem2-3} in this setting.

\begin{lemma}\label{lem4-1}
Let $\gamma_*=d(\bar\varrho_*, \mc O)$. 
There exists $0<\alpha_*<\gamma_*$, such that
\begin{align*}
\lim_{N\to\infty} \max_{\eta^N\in\mc A_i^N} \bb P_{\eta^N}\left( H_N^{\mc O} \le e^{Nh_*} \right)  =  0 ,
\end{align*}
where
\begin{align*}
\mc A_*&:= \mc B(\alpha_*, \bar\varrho_*)\\
2h_*&:=\inf \{ V_*(\varrho): \varrho\notin \mc C_*=\mc C(\gamma_*,\bar\varrho_*) \} > 0,\\
V_*(\varrho) &:=  \inf \left\{ I_T(\pi|\bar\rho_*): T>0, \pi\in D([0,T], \mc M_+), \pi_T=\varrho \right\} , \quad \varrho\in\mc M_+ .
\end{align*}
\end{lemma}

For two real-valued sequences $a_N$, and $b_N$, we denote $a_N \ll b_N$ if $a_N/b_N\to0$ as $N\to\infty$.
Recall the definition of the mixing time $t^N_\mix(\e)$.
Assume that for some $a>0$
\begin{enumerate}
\item[\hypertarget{MT}{(MT)}] $t_\mix^N(1/4) \ll N^a$.
\end{enumerate}
As mentioned before, in \cite{tanaka2020glauberexclusion}, it has been shown that,
for any attractive reaction--diffusion model  with a strictly convex potential
we have $t_\mix^N(1/4)=O(\log N)$, where $O$ denotes the Bachmann--Landau notation. Therefore in this case \hyperlink{MT}{(MT)} is satisfied.

We also need the static large deviation principle for empirical measures.
More precisely, we assume the following static large deviation principle (SLDP).
\begin{enumerate}
\item[\hypertarget{SLDP}{(SLDP)}]  The sequence of probability measures $\{\mc P_N\}_N$ on $\mc M_+$ satisfies 
a large deviation principle with speed $N$ and rate function $V_*$.
Namely, for each closed set $\mc K\subset\mc M_+$,
\begin{align*}
\limsup_{N\to\infty} \dfrac1N \log \mc P_N\left(\mc K \right) \le -\inf_{\varrho\in\mc K}V_*(\varrho),
\end{align*}
and for each open set $\mc U\subset\mc M_+$,
\begin{align*}
\liminf_{N\to\infty} \dfrac1N \log \mc P_N\left(\mc U \right) \ge -\inf_{\varrho\in\mc U}V_*(\varrho).
\end{align*}
\end{enumerate}
Under condition \hyperlink{C}{(C)} in Remark \ref{rem2-1} and \hyperlink{UM}{(UM)}, \hyperlink{SLDP}{(SLDP)} has been established in \cite{MR3947320}.
Note that we use the upper bound of \hyperlink{SLDP}{(SLDP)} only in the following argument.

Before stating the next result, let us return to Example \ref{exa2-1}
to see an example that satisfies the conditions {\rm \hyperlink{UM}{(UM)}, \hyperlink{MT}{(MT)}},
and {\rm \hyperlink{SLDP}{(SLDP)}}.

\begin{example}\label{exa4-1}
Recall the jump rate given in Example \ref{exa2-1} for $0\le\gamma<1$, and define
\begin{align*}
c(\eta)=1+\gamma(1-2\eta(0))(\eta(1)+\eta(-1)-1)+\gamma^2(2\eta(-1)-1)(2\eta(1)-1).
\end{align*}
{\rm \hyperlink{UM}{(UM)}} holds if, and only if, $0\le\gamma\le 1/2$,
{\rm \hyperlink{C}{(C)}} holds if, and only if, $0\le\gamma\le 1/2$ and $V$ is strictly convex if $0 \le\gamma< 1/2$.
Thus, the jump rate $c$ satisfies {\rm \hyperlink{UM}{(UM)}}, {\rm \hyperlink{MT}{(MT)}}, and {\rm \hyperlink{SLDP}{(SLDP)}} for any $0\le\gamma<1/2$.
\end{example}

Our second main result is as follows.

\begin{theorem}\label{thm4-1}
Assume the conditions {\rm \hyperlink{UM}{(UM)}, \hyperlink{MT}{(MT)}} and {\rm \hyperlink{SLDP}{(SLDP)}}. 
Then, for any sequence $\eta^N\in \mc A_*^N$,
$H_N^{\mc O}/\bb E_{\mu_N}[H_N^{\mc O}]$ under $\bb P_{\eta^N}$ converges in distribution
to a mean one exponential random variable.
\end{theorem}

Theorem \ref{thm4-1} follows from a general result established in \cite{MR3131505}.
We first state their result alongside our setting and then prove Theorem \ref{thm4-1}.

Recall the definition of the generator $L_N$.
Define the jump rates $R_N(\eta,\xi), \eta,\xi\in X_N$ using the formula
\begin{align*}
L_Nf(\eta) = \sum_{\xi\in X_N}R_N(\eta,\xi)[f(\xi)-f(\eta)], \quad f:X_N\to\R.
\end{align*}
Let $A_N$ be a sequence of subsets of $X_N$ and $H_{A_N}$ be the hitting time of $A_N$
\begin{align*}
H_{A_N}=\inf\{ t\ge0: \eta_t^N\in A_N\}. 
\end{align*}
Denote  the average rate at which the process jumps from $A_N^c$ to $A_N$ by $r_N(A_N^c, A_N)$
\begin{align*}
r_N(A_N^c, A_N)=\dfrac{1}{\mu_N(A_N^c)}\sum_{\xi\in A_N^c}\mu_N(\xi)R_N(\xi, A_N),
\end{align*}
where $R_N(\xi,A_N)=\sum_{\zeta\in A_N}R_N(\xi,\zeta)$.

The following result has been established in \cite{MR3131505}.

\begin{theorem}\cite[Corollary 1.2]{MR3131505}\label{thm4-2}
Let $A_N$ be a sequence of subsets of $X_N$ such that
\begin{align}
&\lim_{N\to\infty}\mu_N(A_N)=0, \label{4-1}\\
&t_\mix^N(1/4) \ll r_N(A_N^c, A_N)^{-1}, \label{4-4}
\end{align}
and there exists a sequence $S_N$ such that
\begin{align}\label{4-2}
t^N_\mix(1/4) \ll S_N \ll \bb E_{\mu_N}[H_{A_N}].
\end{align}
Further, let $\{\nu_N\}_N$ be a sequence of probability measures on $X_N$ such that
\begin{align}\label{4-3}
\lim_{N\to\infty} \bb P_{\nu_N}\left[ H_{A_N}< S_N \right]=0.
\end{align}
Then, $H_{A_N}/\bb E_{\mu_N}[H_{A_N}]$ under $\bb P_{\nu_N}$ converges
in distribution to a mean one exponential random variable.
\end{theorem}

We can now prove Theorem \ref{thm4-1}.

\begin{proof}[Proof of Theorem \ref{thm4-1}]
For Theorem \ref{thm4-2}, to prove Theorem \ref{thm4-1},
it is enough to show \eqref{4-1}-\eqref{4-3} in the case $A_N=\mc O^N$ and $\nu_N=\delta_{\eta^N}$
for a given sequence $\eta^N\in \mc A_*^N$.

Fix any sequence $\eta^N\in \mc A_*^N$.
Note that \eqref{4-1} is nothing but \eqref{4-5},
and \eqref{4-3} with $S_N=e^{(h_*/2)N}$ holding by Lemma \ref{lem4-1}.
Moreover, the lower bound of \eqref{4-2} with $S_N=e^{(h_*/2)N}$ is clear by \hyperlink{MT}{(MT)}.
The upper bound of \eqref{4-2} can be computed as
\begin{align*}
\bb E_{\mu_N}\left[H_N^{\mc O} \right]
&= \sum_{\xi\in X_N} \bb E_{\xi}\left[H_N^{\mc O} \right] \mu_N(\xi) \\
&\ge \sum_{\xi\in\mc A_*^N} \bb E_{\xi}\left[H_N^{\mc O} {\bf 1}\left\{H_N^{\mc O}> e^{Nh_*}\right\} \right] \mu_N(\xi)\\
&\ge e^{Nh_*}\sum_{\xi\in\mc A_*^N} \bb P_{\xi}\left(H_N^{\mc O}> e^{Nh_*}\right) \mu_N(\xi) \ge \dfrac12e^{Nh_*}\mu_N\left(\mc A_*^N\right).
\end{align*}
In the last inequality, we have used
\begin{align*}
\bb P_{\xi}\left(H_N^{\mc O}> e^{Nh_*}\right) \ge \dfrac12,
\end{align*}
for $N$  sufficiently large and any $\xi\in \mc A_*^N$
(this bound follows from Lemma \ref{lem4-1}).
By Lemma \ref{lem2-2}, we have $\mu_N(\mc A_*^N)\ge 1/2$ for
$N$ sufficiently large.
Thus, we have shown the upper bound of \eqref{4-2}.
By \hyperlink{MT}{(MT)}, to conclude the proof, it is enough to show that
there exists $b>0$ such that
\begin{align}\label{4-6}
r_N([\mc O^N]^c, \mc O^N) \le e^{-bN},
\end{align}
for $N$ sufficiently large. 

Let us prove \eqref{4-6} for some $b>0$.
Let $\partial \mc O^N$ be the outer boundary of $\mc O^N$:
\begin{align*}
\partial\mc O^N=\left[ \bigcup_{x\in\T_N} \left\{\eta\notin \mc O^N: \eta^{x,x+1}\in \mc O^N \right\}\right]
\cup \left[ \bigcup_{x\in\T_N} \left\{\eta\notin \mc O^N: \eta^x\in \mc O^N \right\} \right].
\end{align*}
Note that for each $\xi\in [\mc O^N]^c$, $R_N(\xi,\mc O^N)=0$ unless $\xi\in\partial\mc O^N$
and that 
\begin{align*}
R_N(\xi,\mc O^N)\le\sum_{\zeta\in X_N}R_N(\xi,\zeta) \le N^3/2+N\|c\|_\infty.
\end{align*}
Then we have
\begin{align*}
r_N(A_N^c, A_N)
=\dfrac{1}{\mu_N([\mc O^N]^c)}\sum_{\xi\in \partial\mc O^N}\mu_N(\xi)R_N(\xi, \mc O^N)
\le N^4 \dfrac{\mu_N(\partial\mc O^N)}{\mu_N([\mc O^N]^c)},
\end{align*}
for $N$ sufficiently large.

By \eqref{4-5}, we have
\begin{align*}
\lim_{N\to\infty}\mu_N([\mc O^N]^c)=1.
\end{align*}
To estimate $\mu_N(\partial\mc O^N)$, let $\mc K$ be the closed set $\mc B(\alpha_*, \bar\varrho_*)^c$.
Then we have $\partial \mc O^N\subset \mc K^N$ for $N$ sufficiently large and by Lemma \ref{lemb-4}, we have
\begin{align*}
\inf_{\varrho\in\mc K}V_*(\varrho) > 0.
\end{align*}
Since
\begin{align*}
\mu_N(\partial\mc O^N) \le \mu_N(\partial\mc K^N) = \mc P_N(\mc K),
\end{align*}
 it follows from the upper bound of \hyperlink{SLDP}{(SLDP)} together with the previous bound that
\begin{align*}
\limsup_{N\to\infty}\dfrac1N\log \mu_N(\partial\mc O^N) \le - \inf_{\varrho\in\mc K}V_*(\varrho).
\end{align*}
Summarizing the previous arguments, we obtain
\begin{align*}
r_N(A_N^c, A_N) \le \exp\left\{-\dfrac{N}{2}\inf_{\varrho\in\mc K}V_*(\varrho)\right\},
\end{align*}
for $N$ sufficiently large. Thus \eqref{4-6} is proven, completing the proof of Theorem \ref{thm4-1}.
\end{proof}

\smallskip\noindent{\bf Acknowledgments.}
The author gratefully acknowledges Professor Tanaka Ryokichi for providing many comments
on an earlier version of this paper, significantly improving the presentation of the paper.
K.T. is supported by JSPS Grant-in-Aid for Early-Career Scientists Grant Number 18K13426, JST CREST Mathematics (15656429).

\bibliographystyle{alpha}
\bibliography{reference}

\appendix

\section{Reaction-diffusion equation}\label{seca}

For reader's convenience, we collect miscellaneous lemmata from \cite{MR3765880,MR3947320} which are used in this paper.
When we need a generalization of existing results, we give a proof for the sake of completeness.

The following standard result is used for proving Proposition \ref{propa-1}.

\begin{lemma}\cite[Lemma 7]{MR3947320}\label{lema-5}
There exists a constant $C_0>0$ such that for any weak solutions
$\rho^{j}$, $j=1,2$, to the Cauchy problem \eqref{hdleq} with the initial
condition $\rho_{0}^{j}$ and for any $t>0$, 
\begin{equation*}
\|\rho_{t}^{1}-\rho_{t}^{2}\|_2 \le
e^{C_0 t}\|\rho_{0}^{1}-\rho_{0}^{2}\|_2 .
\end{equation*}
\end{lemma}

The following proposition is a generalization of a part of \cite[Lemma 8]{MR3947320}.
If we take $\rho_0$ as a stationary solution to the Cauchy problem \eqref{hdleq},
\cite[Lemma 8]{MR3947320} can be recovered.

\begin{proposition}\label{propa-1}
Let $\rho_0:\T\to[0,1]$ be a measurable function and $\rho_t(\theta)=
\rho(t,\theta)$ be the unique weak solution to the Cauchy problem
\eqref{hdleq} with the initial condition $\rho_0$.
For any $\beta>0$ and $T>0$, there exists
$0<\beta_0<\beta$, depending only on $\beta$ and $T$,
such that for any measurable function $\tilde\rho_{0}:\T\to[0,1]$ with
$\tilde\rho_{0}(\theta)d\theta\in\mc{B}[\beta_0;\rho_0]$,
we have $\tilde\rho(t,\theta) d\theta\in \mc{B}[\beta;\rho_t]$ for
all $0\le t \le T$, where $\tilde\rho(t,\theta)$ is a unique
weak solution to the Cauchy problem \eqref{hdleq} with the initial
condition $\tilde\rho_{0}$.
\end{proposition}

\begin{proof}
Fix $\beta>0$ and $T>0$. Let $\rho_0:\T\to[0,1]$ be a measurable function and $\rho_t(\theta)=
\rho(t,\theta)$ be the unique weak solution of the Cauchy problem
\eqref{hdleq} with the initial condition $\rho_0$.
Recall the definition of the complete orthogonal normal
basis $\{ e_k; k\in\Z\}$ introduced before \eqref{2-9}.

Because $\rho_t(\theta)$ is a weak solution to the Cauchy problem \eqref{hdleq},
for any weak solution $\tilde\rho_t(\theta)=\tilde\rho(t,\theta)$ to the Cauchy problem
\eqref{hdleq}, with the initial condition $\tilde\rho_0$, we have
\begin{equation*}
d(\tilde\rho_t, \rho_t)  \le  d(\tilde\rho_0, \rho_0) +
\sum_{k\in\bb Z} \frac 1{2^{|k|}}\, \left| 
\dfrac{1}{2} \int_{0}^{t} ds \, \lan [\tilde\rho_s-\rho_s], \Delta e_k \ran 
+ \int_{0}^{t} ds \, \lan [F(\tilde\rho_s) -F(\rho_s)], e_k \ran\right|  .
\end{equation*}
The first term of the right-hand side is bounded by
$\beta/2$ if $\tilde\rho_0 \in \mc{B}[\beta/2;\rho_0]$.
However, the sum is less than or equal to
\begin{equation*}
t \sum_{k\in\bb Z} \frac 1{2^{|k|}}\, \left\{ (2\pi k)^2 + 2\|F\|_\infty \right\}
 =:  C_1 \,t ,
\end{equation*}
because $\rho_s, \tilde\rho_s$ are bounded by $1$, $F$ is bounded, and $\| e_k\|_2 =1$.
Hence, if we set $T_1=\beta/2C_1$, we have 
\begin{equation}\label{3-9}
\tilde\rho_t(\theta)d\theta  \in  \mc B[\beta;\rho_t]  ,
\end{equation}
for any $\tilde\rho_0(\theta)d\theta \in \mc{B}[\beta/2;\rho_0]$ and any $0\le t\le T_1$,

Let $P_t$ be the semigroup on $L^2(\T)$ generated by $(1/2)\Delta$.
Then, by Duhamel's formula,  we have
\begin{align}\label{3-5}
\|\tilde\rho_t- \rho_t\|_{2}
& \le  \| P_t(\tilde\rho_{0} - \rho_0)\|_{2}+
\int_{0}^{t} \left \| P_{t-s} \left[F(\tilde\rho_s)-F(\rho_s)\right]  \right\|_{2} ds \notag \\
& \le   \| P_t ( \tilde\rho_0-\rho_0) \|_2  +  t\|F'\|_{\infty}  .
\end{align}
Let $T_2=\min \{T, T_1, \beta/(6e^{C_0T}\|F'\|_\infty)\}$.

Because $P_t(\tilde\rho_0-\rho_0)$ is a solution to the heat equation,
there exists some $0 < \beta_0 < \beta/2$,
depending only on $\beta$ and $T$,
such that for any $\tilde\rho_0(\theta)d\theta\in \mc B[\beta_0;\rho_0]$
\begin{align}\label{3-6}
\| P_{T_2} ( \tilde\rho_0-\rho_0) \|_2  \le  \beta/(6e^{C_0T}) .
\end{align}
See the paragraph after (3.10) of \cite{MR3947320} for details.
Let $C_0$ be the constant appearing in Lemma \ref{lema-5}.
Then, by \eqref{3-5}, \eqref{3-6}, and Lemma \ref{lema-5}, we have
\begin{align}\label{3-7}
\|\tilde\rho_t - \rho_t \|_2  \le  e^{C_0(t-T_2)}\|\tilde\rho_{T_2} - \rho_{T_2}\|_2 \le \beta/3 ,
\end{align}
for any $\tilde\rho_0(\theta)d\theta\in \mc B[\beta_0;\rho_0]$ and $T_2\le t \le T$.

Therefore, it follows from $T_2 \le T_1$, \eqref{2-17}, \eqref{3-9}, and \eqref{3-7} that
$\tilde\rho_t(\theta)d\theta\in\mc B[\beta;\rho_t]$
for any $0 \le t \le T$ provided $\tilde\rho_0(\theta)d\theta\in\mc B[\beta_0;\rho_0]$,
which completing the proof of Proposition \ref{propa-1}.
\end{proof}

Because of Proposition \ref{propa-1}, we can obtain the following corollary.

\begin{corollary}\label{cora-1}
Under the notations of the proof of Lemma \ref{lem3-2}, 
the application $\rho(\theta)d\theta\in\mc B_i \cap \mc M_{+,1}\mapsto\tau^{(i)}(\rho)$
is lower semicontinuous with respect to the weak topology.
Namely, for any fixed $\rho(\theta)d\theta\in \mc B_i\cap \mc M_{+,1}$ and
for any sequence $\{\rho_n(\theta)d\theta\}_n$ in $\mc B_i\cap \mc M_{+,1}$,
which converges to $\rho(\theta)d\theta$ in the weak topology, we have
\begin{align*}
\tau^{(i)}(\rho) \le \liminf_{n\to\infty} \tau^{(i)}(\rho_n).
\end{align*}
\end{corollary}
\begin{proof}
This corollary is a direct consequence of Proposition \ref{propa-1}.
To see this, take any $\rho(\theta)d\theta\in \mc B_i\cap \mc M_{+,1}$ and
any sequence $\{\rho_n(\theta)d\theta\}_n$ in $\mc B_i\cap \mc M_{+,1}$
converging to $\rho(\theta)d\theta$ in the weak topology. We can assume
the loss of generality that $\tau^{(i)}(\rho)>0$.

Let $\rho(t)$ and  $\rho_n(t)$ be the unique weak solutions to the Cauchy problem \eqref{hdleq}
with the initial conditions $\rho$ and $ \rho_n$, respectively.
Since $\mc B[\alpha_i;\bar\varrho_i]$ is closed, it follows from the definition of $\tau^{(i)}(\rho)$ that
$\rho(\tau^{(i)}(\rho))\in \mc B[\alpha_i;\bar\varrho_i]$ and $\rho(t)\notin \mc B[\alpha_i;\bar\varrho_i]$
for any $0 \le t  < \tau^{(i)}(\rho)$.

Fix the small $\e>0$ and let $t_\e=\tau^{(i)}(\rho)-\e>0$. Since $\mc B[\alpha_i;\bar\varrho_i]$ is closed and $\rho(t_\e)\notin \mc B[\alpha_i;\bar\varrho_i]$,
we have
\begin{align*}
\alpha_\e:= \inf \{ d(\rho(t_\e), \tilde\rho) : \tilde\rho\in \mc B[\alpha_i;\bar\varrho_i]\} >0.
\end{align*}
By applying Proposition \ref{propa-1} for $\beta=\alpha_\e/2$ and $T=t_\e$, there exists $\beta_\e<\alpha_\e/2$ such that
for any measurable function $\tilde\rho_{0}:\T\to[0,1]$ with
$\tilde\rho_{0}(\theta)d\theta\in\mc{B}[\beta_\e;\rho_0]$,
we have $\tilde\rho(t,\theta) d\theta\in \mc{B}[\alpha_\e/2;\rho_t]$ for
all $0\le t \le t_\e$, where $\tilde\rho(t,\theta)$ is a unique
weak solution to the Cauchy problem \eqref{hdleq} with the initial
condition $\tilde\rho_{0}$. Then we have $\rho_n(t)\notin \mc B[\alpha_i;\bar\varrho_i]$
for any $0 \le t  \le t_\e$ if $\rho_n(\theta)d\theta\in\mc{B}[\beta_\e;\rho_0]$.
Therefore, for any large enough $n$,  we have
\begin{align*}
t_\e \le \tau^{(i)}(\rho_n).
\end{align*}
Taking $n\to\infty$ and $\e\to0$ completes the proof of Corollary \ref{cora-1}.
\end{proof}

Recall the definitions of $\rho_i, \bar\rho_i$, and $\bar\varrho_i$ from Section \ref{sec2}.
Note that $x\in(0,1)$ is a limit point of the dynamical system
\begin{align*}
\dfrac{d}{dt}x_t = - F(x_t),
\end{align*}
if, and only if, $x=\rho_i, i=1,\ldots,\ell$.
The following result shows that the constant function $\bar\rho_i$
is a local attractor of the dynamical system defined by \eqref{hdleq}
with respect to the weak topology.

\begin{lemma}\cite[Lemma 11]{MR3947320}\label{lema-3}
Let $\varepsilon>0$ and $i=1,\ldots,\ell$.
There exists $\gamma_i>0$ such
that for any density condition $\rho: \bb T \to [0,1]$ such that
$\rho(\theta) d\theta \in \mc B(\gamma_i;\bar\varrho_i)$,
$\rho_t$ converges in the supremum norm to $\bar\rho_i$, as
$t\to\infty$, where $\rho_t(\theta) = \rho(t,\theta)$ is a
unique weak solution to the Cauchy problem \eqref{hdleq} with the initial condition
$\rho$. Moreover, $\pi_t(d\theta) = \rho(t,\theta) d\theta$ belongs
to $\mc B(\e;\bar\varrho_i)$ for all $t\ge 0$.
\end{lemma}

Lemma \ref{lema-3} immediately implies the following result.

\begin{corollary}\label{cora-2}
For each $i=1,\ldots,\ell$, we have
\begin{align*}
\inf \left\{ d(\bar\varrho_i, \bar\varrho): \bar\varrho\in\mc M_\sol, \bar\varrho\neq\bar\varrho_i \right\} > 0.
\end{align*}
\end{corollary}

\section{Dynamical rate function}\label{secb}

In this appendix, we collect miscellaneous lemmata regarding the rate function of the dynamical large deviation principle

\begin{lemma}\cite[Proposition 4.1]{MR3765880}\label{lemb-5}
Fix $T>0$ and a measurable function $\rho:\T\to[0,1]$.
Let $\pi$ be a trajectory in $D([0,T], \mc M_+)$ such that $I_T(\pi|\rho)$ is finite.
Then $\pi$ belongs to $C([0,T],\mc M_{+,1})$ and $\pi(0,d\theta)=\rho(\theta)d\theta$.
\end{lemma}

\begin{lemma}\cite[Corollary 4.6]{MR3765880}\label{lemb-1}
Fix $T>0$. The density $\rho$ of a trajectory $\pi(t,d\theta)=\rho(t,\theta)d\theta$ in
$D([0,T], \mc M_{+,1})$ is the weak solution to the Cauchy problem
\eqref{hdleq} with initial condition $\rho_0$ if, and only if, $I_{T}(\pi|\rho_0)
= 0$.
\end{lemma}

Recall the definition of $\mb D_{T,\beta}$ defined before \eqref{3-17}.

\begin{lemma}\cite[Lemma 14]{MR3947320}\label{lemb-3}
For each $\beta>0$ there exists $T=T(\beta)>0$ such that
\begin{align*}
\inf_{\pi\in \mb D_{T,\beta}}I_T(\pi) >0.
\end{align*}
In particular, 
for each $\beta>0$ and each $A>0$ there exists $T=T(\beta,A)>0$ such that
\begin{align*}
\inf_{\pi\in \mb D_{T,\beta}}I_T(\pi) \ge A.
\end{align*}
\end{lemma}

The first assertion of Lemma \ref{lemb-3} is proven in \cite[Lemma 14]{MR3947320}, whereas 
the second assertion is a direct consequence of the first one.

Recall the definitions of $\mc M_\sol(\alpha)$ and $\tilde H_N(\alpha)$, which are defined after \eqref{3-16}.

\begin{lemma}\cite[Lemma 21]{MR3947320}\label{lemb-2}
For each $\alpha>0$, there exist $T_0, C_0, N_0 >0$, depending on
$\alpha$, such that, for all $N\ge N_0$ and all $k\ge 1$,
\begin{equation*}
\sup_{\eta\in X_N} \bb P_{\eta}\left[\tilde H_N(\alpha) \ge kT_0\right] \le e^{ -k C_0 N}.
\end{equation*}
\end{lemma}

\begin{lemma}\cite[Lemma 30]{MR3947320}\label{lemb-4}
For each $i=1,\ldots,\ell$ and each $\alpha>0$, we have
\begin{equation*}
\inf \left\{ V_i(\varrho): \varrho \not\in \mc
B(\alpha;\bar\varrho_i) \right\} > 0.
\end{equation*}
\end{lemma}

\end{document}